\documentclass[12pt,twoside]{amsart}
\usepackage{amsmath}
\usepackage{amssymb}
\usepackage{amscd}
\usepackage{xypic}
\usepackage{makecell}

\usepackage{booktabs}

\xyoption{all}
\usepackage{todonotes}

\title[Conjugacy classes of elements of finite order in $p$-compact groups]{Counting conjugacy classes of elements of finite order in $p$-compact groups}

\author{Jos\'e Cantarero}
\thanks{}
\address{
\hfill\break Centro de Investigaci\'on en Matem\'aticas, A.C. Unidad M\'erida \\
\hfill\break Parque Cient\'ifico y Tecnol\'ogico de Yucat\'an \\ 
\hfill\break Carretera Sierra Papacal--Chuburn\'a Puerto Km 5.5 \\
\hfill\break Sierra Papacal, M\'erida, YUC 97302 \\
\hfill\break M\'exico
\hfill\break {\emph{Email address: }}{\tt cantarero@cimat.mx}}

\author{Bernardo Villarreal}
\thanks{}
\address{
\hfill\break Departamento de Matem\'aticas, Centro de Investigaci\'on y de Estudios Avanzados del Instituto Polit\'ecnico Nacional\\
\hfill\break Av. IPN 2508\\
\hfill\break  07360, CDMX\\
\hfill\break M\'exico
\hfill\break {\emph{Email address: }}{\tt bvillarreal@math.cinvestav.mx}}

\newcommand{\comments}[1]{}

\newcommand{\Aut}{\operatorname{Aut}\nolimits}
\newcommand{\cc}{\operatorname{cc}\nolimits}

\newcommand{\Coker}{\operatorname{Coker}\nolimits}

\newcommand{\ev}{\operatorname{ev}\nolimits}
\newcommand{\Ev}{\operatorname{Ev}\nolimits}
\newcommand{\Ext}{\operatorname{Ext}\nolimits}

\newcommand{\GL}{\operatorname{GL}\nolimits}

\newcommand{\Hom}{\operatorname{Hom}\nolimits}

\newcommand{\im}{\operatorname{Im}\nolimits}

\newcommand{\Ker}{\operatorname{Ker}\nolimits}
\newcommand{\Map}{\operatorname{Map}\nolimits}

\newcommand{\Rep}{\operatorname{Rep}\nolimits}
\newcommand{\rk}{\operatorname{rk}\nolimits}
\newcommand{\sgn}{\operatorname{sgn}\nolimits}

\newcommand{\spann}{\operatorname{span}\nolimits}
\newcommand{\Sq}{\operatorname{Sq}\nolimits}

\newcommand{\Tor}{\operatorname{Tor}\nolimits}

\newcommand{\Mod}[1]{\ \mathrm{mod}\ #1}

\def \A{{\mathcal A}}

\def \C{{\mathbb C}}
\def \F{{\mathbb F}}

\def \Q{{\mathbb Q}}

\def \Z{{\mathbb Z}}
\def \D{{\bf D}}

\newcommand{\Ff}{{\mathcal{F}}}
\newcommand{\Ll}{{\mathcal{L}}}
\newcommand{\pcom}{^\wedge_p}

\theoremstyle{plain}
\newtheorem*{introtheoremA}{Theorem A}
\newtheorem*{introtheoremB}{Theorem B}
\newtheorem*{introtheoremC}{Theorem C}
\newtheorem{theorem}{Theorem}[section]
\newtheorem{proposition}[theorem]{Proposition}
\newtheorem{corollary}[theorem]{Corollary}
\newtheorem{lemma}[theorem]{Lemma}

\theoremstyle{definition}
\newtheorem{definition}[theorem]{Definition}
\newtheorem{remark}[theorem]{Remark}



\begin{document}

\begin{abstract}
We express the set of representations from a cyclic $p$-group to a connected
$p$-compact group in terms of the associated reflection group and compute its
cardinality for each exotic $p$-compact group.
\end{abstract}

\maketitle


\section*{Introduction}

There is a deep connection between a group and its classifying space, 
specially for finite and compact Lie groups. Homotopical group 
theory was born from the idea that group theory can be done at 
the level of classifying spaces, and this idea has materialized 
in several successful theories which study new objects, such as
$p$-local finite groups, $p$-compact groups and $p$-local compact
groups.

In this paper we will focus on $p$-compact groups, which were introduced by Dwyer and 
Wilkerson in \cite{DW} to determine cohomological properties of finite loop spaces. In 
Section \ref{SectionHomomorphisms} we review the concepts of the theory of $p$-compact 
groups that are needed in the paper, but for this introduction it suffices to say that
they are $\F_p$-finite loop spaces of pointed, connected and $\F_p$-complete spaces.

The structure and properties of $p$-compact groups are remarkably similar to those of 
compact Lie groups. For instance, isomorphism classes of connected $p$-compact groups 
are in bijective correspondence with isomorphism classes of root data over $\Z \pcom$, 
which led to their classification in \cite{AGMV} and \cite{AG}. We direct the interested 
reader to \cite{Gro} for a panoramic view of the theory. 

Properties of compact Lie groups which can be expressed in terms of their $p$-completed 
classifying spaces often have a version in the theory of $p$-compact groups. For example, if $P$
is a finite $p$-group and $G$ is a connected compact Lie group (see \cite[Theorem 1.1]{DW} and \cite[Theorem 0.4]{L}), there is a bijection
\[ \Rep(P,G) \cong [BP,BG \pcom]. \]
In particular, conjugacy classes of elements $x \in G$ such that $p^n x = 0$ are in
bijective correspondence with $[B\Z/p^n,BG \pcom]$. There has been a renewed interest 
(\cite{FS2}, \cite{FS}) in the number of conjugacy classes of homomorphisms from cyclic 
groups to compact Lie groups and related numbers due to its connection with the number 
of vacua in the quantum moduli space of $M$-theory compactifications on manifolds 
of $G_2$ holonomy. Certain relationships found between these numbers in \cite{FS} were 
found in \cite{FS2} to have physical implications. 

In the language of $p$-compact groups, a homomorphism $f \colon X \to Y$ is a pointed map 
$Bf \colon BX \to BY$ and two homomorphisms $f$, $g$ are conjugate if $Bf$ and $Bg$ are freely homotopic. Since finite $p$-groups are $p$-compact
groups, in this language $[B\Z/p^n,BX]$ correspond to conjugacy classes of homomorphisms
$\Z/p^n \to X$, hence this is an appropriate generalization of $\Rep(\Z/p^n,G)$. Any connected 
$p$-compact group is isomorphic to a unique product of the form $G \pcom \times Z$, where $G$ is a connected compact Lie group 
and $Z$ is a finite product of exotic $p$-compact groups. The computation of the size of $\Rep(\Z/p^n,G)$
was treated in \cite{DP}, \cite{Dj1}, \cite{Dj2} and \cite{PW}, hence in this article we focus on computing the
cardinality of $[B\Z/p^n, BX]$ when $X$ is an exotic $p$-compact group.
 
The connected $p$-compact groups with associated $\Z \pcom$-reflection groups $(W,L)$ such that 
$L \otimes_{\Z \pcom} \Q \pcom$ is an irreducible representation of $W$ are called simple, and they are organized
in four infinite families and $34$ exceptional cases. Exotic $p$-compact groups are simple $p$-compact
groups which do not correspond to a compact Lie group. They are called modular if $p$ divides the
order of $W$ and non-modular otherwise. The only modular exotic $p$-compact groups are generalized
Grassmanians $X(m,s,n)$ in the family 2a with $m>2$, and the exceptional cases $X_j$ with $j \in \{ 12,24,29,31,34\}$.

In Section \ref{SectionHomomorphisms}, for a connected $p$-compact group $X$, we establish bijections between $[B\Z/p^n,BX]$ and certain sets built 
from the action of its Weyl group on its maximal torus. For instance, if $(W,L)$ is the $\Z \pcom$-reflection
group associated to $X$, then Corollary \ref{BijectionWithLattice} shows that there is a bijection
\[ \frac{L/p^n L}{W} \cong [B\Z/p^n,BX], \]
and in particular this is a finite set. Using this bijection and Burnside's counting formula, we can determine the size of 
$[B\Z/p^n,BX]$ from the cardinalities of the fixed points of the elements of $W$ for its action on $L/p^n L$. In Section \ref{NonModularCases},
we show that if $g \in W$ belongs to a reflection subgroup of order prime to $p$, then these fixed points are just the mod $p^n$ reduction of
the fixed points of the action on $L$. In the non-modular case, this holds for all elements of $W$ and a result of Solomon expresses Burnside's 
counting formula in terms of the exponents of $W$ as a $\Z\pcom$-reflection group.

\begin{introtheoremA}
If $X$ is a non-modular connected $p$-compact group with exponents $m_i$, then
\[ |[B\Z/p^k,BX]| = \prod_{i=1}^l \frac{m_i+p^k}{m_i+1} \]
for all $k \geq 1$.
\end{introtheoremA}

For exotic generalized Grassmanians in the family 2a, it is more convenient to use the bijection
\[ [B\Z/p^n,BX] \cong \Omega_{p^k}(\hat{T})/W, \]
also shown in Section \ref{SectionHomomorphisms}. Here $\hat{T}$ is a discrete approximation 
to the maximal torus of $X$ and $\Omega_{p^k}(\hat{T})$ is the subgroup of elements of $\hat{T}$ 
with order dividing $p^k$. In Section \ref{generalizedgrassmanians} we determine a fundamental 
domain for the action of $W$ on $\Omega_{p^k}(\hat{T})$, in the case when $X$ is a generalized 
Grassmanian in the family 2a not coming from a compact Lie group, and count its number of elements.

\begin{introtheoremB}
If $X(m,s,n)$ belongs to the family 2a with $m >2$, we have 
\[ | [B\Z/p^k,BX(m,s,n)] | = 1 + \frac{p^k-1}{m} s + \sum_{j=1}^{p^k-1/m} {{n-2+j} \choose {j}} \left( \frac{p^k-1}{m} - j + 1 \right) s \]
for all $k \geq 1$.
\end{introtheoremB}

Four of the remaining five cases are treated individually in Section \ref{ModularCases}. For each of these $p$-compact
groups, there exist elements such that their fixed points on $L/p^n L$ are not the mod $p^n$ reduction of the
fixed points of the action on $L$. But we find in each case enough non-modular reflection subgroups so that many
elements satisfy this condition, and treat the rest of elements by hand. This is particularly useful for $X_{29}$ and $X_{31}$,
since the Weyl groups of $X_{12}$ and $X_{24}$ are small enough to list representatives of their conjugacy classes and compute
the fixed points for each of them. Finally, the computation for the $7$-compact group $X_{34}$ is achieved using GAP \cite{GAP}.

\begin{introtheoremC}
The following formulas hold for all $k \geq 1$.
\begin{align*}
|[B\Z/3^k,BX_{12}]| & = \frac{1}{48} ( 3^{2k} + 12 \cdot 3^k + 51 ), \\
|[B\Z/2^k,BX_{24}]| & = \frac{1}{336}(2^{3k} + 21 \cdot 2^{2k} + 140 \cdot 2^k + 216 + 42 \cdot 2^{\min\{k,2\}}), \\
|[B\Z/5^k,BX_{29}]| & = \frac{1}{7680}(5^{4k}+40 \cdot 5^{3k} + 530 \cdot 5^{2k} + 2720 \cdot 5^k + 5925), \\
|[B\Z/5^k,BX_{31}]| & = \frac{1}{46080}(5^{4k}+60 \cdot 5^{3k} + 1270 \cdot 5^{2k} + 11100 \cdot 5^k + 42865), \\
|[B\Z/7^k,BX_{34}]| & = \frac{1}{39191040}(7^{6k}+ a_5 \cdot 7^{5k} + a_4 \cdot 7^{4k} + a_3 \cdot 7^{3k} +a_2 \cdot 7^{2k} + a_1 \cdot 7^k + a_0),
\end{align*}
where $a_5=126$, $a_4=6195$, $a_3=151060$, $a_2=1904679$, $a_1=11559534$ and $a_0=31168165$.
\end{introtheoremC}

We observed that if $X$ is an exotic $p$-compact group corresponding to an exceptional
finite reflection group $W_X(T) \leq GL_n(\Z \pcom)$ and $w$ belongs to a reflection subgroup 
$H$ of $W_X(T)$, then the order of the torsion subgroup of $\Coker(w-1)$ divides 
the order of the $p$-Sylow subgroup of $H$. We do not know if this holds for generalized
Grassmanians since our computation method for those cases did not explicitly find these
cokernels. It would be interesting to know whether this holds for all exotic $p$-compact
groups.
\newline

\noindent \begin{bf}Acknowledgments: \end{bf} Project supported by CONAHCYT (now Secihti) 
in the year 2023 under Frontier Science Grant CF-2023-I-2649.

\section{Homomorphisms from cyclic $p$-groups to $p$-compact groups}
\label{SectionHomomorphisms}

In this section we review some concepts from the theory of $p$-compact
groups and express the number of homomorphisms from cyclic $p$-groups 
to $p$-compact groups in terms of the maximal torus and the Weyl group.

Recall that a $p$-compact group is a triple $(X,BX,e \colon X \to \Omega BX)$, where $X$ is an $\F_p$-finite
space, $e$ is a homotopy equivalence and $BX$ is a pointed, connected and $\F_p$--complete space in the sense of 
Bousfield-Kan \cite{BK}. Even though the triple is determined by $BX$, we will use $X$ to refer to it. For instance, 
if $G$ is a compact Lie group such that $\pi_0(G)$ is a finite $p$-group, then $G \pcom$ is a $p$-compact group.
These objects were introduced in \cite{DW} and there is a classification theorem (\cite{AGMV},\cite{AG}) that states that any connected $p$-compact group $X$ 
is isomorphic to a unique product of the form $G \pcom \times Z$, where $Z$ is a finite product of exotic $p$-compact 
groups.

A homomorphism $f \colon X \to Y$ of $p$-compact groups is a pointed map $ Bf \colon BX \to BY$. The centralizer $C_Y(f(X))$ 
of $f(X)$ in $Y$ is the $p$-compact group $\Omega \Map(BX,BY)_f$. A $p$-compact torus $T$ of rank $r$ is the loop space 
of an Eilenberg-MacLane space $K( (\Z \pcom)^r,2 )$. Any homomorphism $T \to X$ from a $p$-compact torus factors 
through the centralizer $C_X(T)$ and we say that the homomorphism is self-centralizing if the map $T \to C_X(T)$
is an equivalence. A maximal torus for a connected $p$-compact group $X$ is a $p$-compact torus $T$ with a 
self-centralizing homomorphism $i \colon T \to X$. Any connected $p$-compact group possesses a maximal torus. 
The Weyl group $W_X(T)$ of $X$ is the group of homotopy classes of homotopy equivalences
$f \colon BT \to BT$ such that $ Bi \circ f \simeq Bi$.
 
The induced action on $\pi_2(BT) \cong (\Z \pcom)^r$ exhibits $W_X(T)$ as a finite reflection
group over $\Z \pcom$. A $\Z \pcom$--root datum can be determined as well, and the classification
theorem gives a bijective correspondence between isomorphism classes of $\Z \pcom$--root data
and isomorphism classes of connected $p$-compact groups. The exotic $p$-compact groups are
those corresponding to finite reflection groups $W \to GL(V)$ over $\Z \pcom$ which do not come from a
finite reflection group over $\Z$, and such that $V \otimes \Q$ is an irreducible representation
of $W$.

As we explained in the introduction, if $G$ is a compact Lie group and $p$ is a prime, there is a bijection
between $[B\Z/p^n,BG \pcom]$ and the set of conjugacy classes of elements $x \in G$ such that $p^n x = 0$. This
motivates our study of the sets $[B\Z/p^n,BX]$ for a $p$-compact group. If $X$ is connected, by the classification theorem
\[ [B\Z/p^n,BX] \cong [B\Z/p^n,BG \pcom] \times \prod_{i=1}^k [B\Z/p^n,BZ_i] \cong \Rep(\Z/p^n,G) \times \prod_{i=1}^k [B\Z/p^n,BZ_i], \]
where the $Z_i$ are exotic $p$-compact groups. Any connected compact Lie group $G$ is isomorphic to the quotient by a finite central subgroup 
$C$ of a product of a simply connected compact Lie group $H$ and a torus $T$. If $p$ and $|C|$ are relatively prime, then \cite[Lemma 1]{Dj1} 
shows that the quotient $H \times T \to G$ induces a bijection
\[ \Rep(\Z/p^n,H \times T) \to \Rep(\Z/p^n,G). \]
If $T \cong (S^1)^r$, it is easy to see that $\Rep(\Z/p^n,T) \cong (\Z/p^n)^r$, so the problem is reduced
to determining $\Rep(\Z/p^n,H)$ for a simply connected compact Lie group $H$. It suffices to determine
$\Rep(\Z/p^n,K)$ for simple, simply connected compact Lie groups, since simply connected compact Lie groups 
are isomorphic to finite products of such groups. The sizes of these sets were computed in \cite{Dj1} and \cite{Dj2} (see also \cite{DP}, \cite{PW} and \cite{FS}), 
hence we will focus on computing the size of $[B\Z/p^n,BZ]$ for exotic $p$-compact groups.

Given a homomorphism $f \colon H \to X$ from an abelian $p$-compact toral group to a $p$-compact group, Proposition 8.2 in \cite{DW} shows
that $f$ lifts to a central homomorphism $f' \colon H \to C_X(H)$. The next lemma shows the naturality of this map.

\begin{lemma}
\label{NaturalCentralizers}
Given an up-to-homotopy commutative diagram
\[ 
\diagram
BH \rto^{Bf_H} \dto_{B\alpha} & BX \\
BK \urto_{Bf_K} &
\enddiagram
\]
of homomorphisms of $p$-compact groups, where $BH$, $BK$ are abelian $p$-compact toral groups, the canonical central maps $BH \to BC_X(H)$ and 
$BK \to BC_X(K)$ fit into an up-to-homotopy commutative diagram
\[ 
\diagram
BH \rto \dto_{\alpha} & BC_X(H) \\
BK \rto & BC_X(K) \uto^{\alpha^*} 
\enddiagram
\]
\end{lemma}

\begin{proof}
The canonical map $BH \to BC_X(H)$ is constructed in \cite[Proposition 8.2]{DW} as the composition
\[ BH \to \Map(BH,BH)_1 \stackrel{(Bf_H)_*}{\longrightarrow} \Map(BH,BX)_{Bf_H}, \]
where the first map is a homotopy inverse for the evaluation at the basepoint. Since $K$ is abelian, the map $\alpha$
is central and therefore $C_K(H) \to K$ is an equivalence. We have a commutative diagram
\[ 
\diagram
\Map(BK,BK)_1 \dto_{B\alpha^*} \rrto_{\simeq}^{\ev_*} & & BK \\
\Map(BH,BK)_{B\alpha} \urrto_{\Ev_*}^{\simeq} & &
\enddiagram
\]
where $\ev_*$ and $\Ev_*$ are evaluations at the base point. We obtain that $B\alpha^*$ is an equivalence. 
Let $\gamma$ and $\beta$ be homotopy inverses for $B\alpha^*$ and $\Ev_*$, respectively, so that we can
choose $\gamma \beta$ as a homotopy inverse for $\ev_*$. Then the following diagram is commutative up
to homotopy
\[ 
\diagram
BH \rto \drrto_{B\alpha} & \Map(BH,BH)_1 \rto^{B\alpha_*} & \Map(BH,BK)_{B\alpha} \rto^{(Bf_K)_*} \dto_{\Ev_*} & \Map(BH,BX)_{Bf_H} \\
 & & BK \rto_{(Bf_K)_* \gamma \beta \qquad} & \Map(BK,BX)_{Bf_K} \uto_{B\alpha^*}
\enddiagram
\]
and the desired result follows.
\end{proof}

\begin{lemma}
\label{ExtensionToWeylGroup}
Let $H$, $K$ be cyclic $p$-subgroups of a discrete approximation $T'$ to the maximal torus $T$ of the $p$-compact group $X$. 
If $\alpha \colon H \to K$ is an isomorphism such that the diagram
\[
\diagram
BH \dto_{B\alpha} \rto & BX \\
BK \urto &
\enddiagram
\]
commutes up to homotopy, then there is a homotopy equivalence $\omega \colon BT \to BT$ such that the diagram
\[
\diagram
BH \dto_{B\alpha} \rto & BT \dto^{\omega} \rto & BX \\
BK \rto & BT \urto &
\enddiagram
\]
commutes up to homotopy.
\end{lemma}

\begin{proof}
Since $H$ and $K$ are finite $p$-groups, they are also $p$-compact groups and we can consider their 
centralizers in $X$. The map $B\alpha^* \colon BC_X(K) \to BC_X(H)$ is a 
homotopy equivalence because $\alpha$ is an isomorphism. Let $B\alpha_*$ be its homotopy inverse. The diagram
\[ 
\diagram
BH \rto \dto_{B\alpha} & BC_X(H) \dto^{B\alpha_*} \\
BK \rto & BC_X(K)
\enddiagram
\]
is commutative up to homotopy by Lemma \ref{NaturalCentralizers}. Since $T'$ is abelian, both horizontal maps factor through $BT$ up to homotopy and we have
a diagram
\[ 
\diagram
BH \rto^{Bj_H} \dto_{B\alpha} & BT \rto^{B\iota_H} & BC_X(H) \dto^{B\alpha_*} \\
BK \rto_{Bj_K} & BT \rto_{B\iota_K} & BC_X(K)
\enddiagram
\]
which commutes up to homotopy. By \cite[Proposition 4.3]{DW2}, the maps $B\alpha_* B\iota_H$ and $B\iota_K$ are both maximal
tori for $BC_X(K)$. By \cite[Proposition 8.11]{DW}, there is a homomorphism $\omega \colon BT \to BT$ such that  $B\alpha_* B\iota_H \simeq B\iota_K \circ \omega $.
Since $Bi \circ \omega \simeq Bi$, we obtain that $\omega$ is a self-homotopy equivalence of $BT$ using \cite[Lemma 9.3]{DW}. 
We can factor further $Bj_H$ and $Bj_K$
\[ 
\diagram
BH \rto^{Bi_H} \dto_{B\alpha} & BT' \rto^a & BT \\
BK \rto_{Bi_K} & BT' \rto_a & BT
\enddiagram
\]
and there is $B\omega' \colon BT' \to BT'$ such that $a B\omega' \simeq \omega a$. Hence we have
\[ B\iota_K Bj_K B\alpha \simeq B\alpha_* B\iota_H Bj_H \simeq B\iota_K \omega Bj_H. \]
The maps $B\omega' Bi_H$ and $Bi_K B\alpha$ satisfy 
\begin{align*} 
B\iota_K a Bi_K B\alpha & \simeq B\iota_K Bj_K B\alpha \\
                        & \simeq B\iota_K \omega Bj_H \\
                        & \simeq B\iota_K \omega a Bi_H \\ 
                        & \simeq B\iota_K a B\omega' Bi_H
\end{align*}
and $B\iota_K a Bi_K B\alpha \simeq B\iota_K Bj_K B\alpha$ is central. By \cite[Lemma 5.4]{DW2}, we obtain that $i_K(\alpha(x))^{-1} \omega'(i_H(x))^{-1}$
belongs to the kernel of $B\iota_K a$ for all $x \in H$. But since $\iota_K$ is a monomorphism, the kernel of $B\iota_K a$ is trivial by \cite[Theorem 7.3]{DW}.
Therefore
\[ Bi_K B\alpha \simeq B\omega' Bi_H, \]
hence
\[ Bj_K B\alpha \simeq \omega Bj_H, \]
as we wanted to show. 
\end{proof}

Given an element $a$ of order $n$ in a $p$-discrete toral group $G$, we use the notation $\kappa_a$ for
the homomorphism $\Z/n \to G$ that sends the class of $1$ to $a$, as in \cite[Section 7]{DW}.

\begin{proposition}\label{prop-fromXtoTW}
Let $T$ be a maximal torus of the connected $p$-compact group $X$. For any cyclic $p$-group $A$ there
is a bijection
\[ [BA,BT]/W_X(T) \to [BA,BX]. \]
\end{proposition}

\begin{proof}
Consider the map $[BA,BT] \to [BA,BX]$ induced by the monomorphism $Bi \colon BT \to BX$. The action
of $W_X(T)$ is through self homotopy equivalences $f$ of $BT$ which satisfy $Bi \circ f \simeq Bi$,
hence we have an induced map
\[ \varphi \colon [BA,BT]/W_X(T) \to [BA,BX]. \]
Given $ h \colon BA \to BX$, by repeated applications of \cite[Proposition 5.6]{DW}, there exists
$ z \colon B\Z/p^{\infty} \to BX$ such that $ z \circ Bj$ is homotopic to $h$, where $j$ is the 
inclusion of $A$ in $\Z/p^{\infty}$. We can extend it further, up to homotopy, to a map $ \overline{z} \colon K(\Z \pcom,2) \to BX$ 
by \cite[Proposition 6.8]{DW}. By \cite[Proposition 8.11]{DW}, there exists $y \colon K(\Z \pcom, 2) \to BT$ 
such that $Bi \circ y \simeq \overline{z}$. The composition $k \colon BA \to BT$ is such that $ Bi \circ k \simeq h$, hence 
$\varphi$ is surjective.

Let $f$, $g \colon BA \to BT$ be such that $Bi \circ g \simeq Bi \circ f$. Since $BT = K((\Z \pcom)^r,2)$ is the $p$-completion of the classifying
space of a torus $T'$ and $[BA,BT] \cong \Rep(A,T')$, there is a homomorphism $f' \colon A \to T'$ such that $f$ is homotopic to the composition of 
$Bf'$ and the $p$-completion map $BT' \to BT$. Hence we can factor $f$ up to homotopy as a composition
\[ BA \stackrel{B\hat{f}}{\longrightarrow} B\im(f') \stackrel{Bj_1}{\longrightarrow} BT, \]
where $\hat{f} \colon A \to \im(f')$ is the restriction of $f'$ to its codomain. Similarly, $g \simeq Bj_2 B\hat{g}$ 
for a certain homomorphism $g' \colon A \to T'$. 

Assume first that $f'$ and $g'$ are injective, so that $\hat{f}$ and $\hat{g}$ are isomorphisms. Then
\[ Bi Bj_2 B(\hat{g}\hat{f}^{-1}) \simeq Bi Bj_1. \]
By Lemma \ref{ExtensionToWeylGroup}, there exists a representative $\omega \colon BT \to BT$ of an element in $W_X(T)$ 
such that $Bj_2 B(\hat{g}\hat{f}^{-1}) \simeq \omega Bj_1$ and therefore
\[ g \simeq Bj_2 B\hat{g} \simeq \omega Bj_1 B\hat{f} \simeq \omega f. \]

To show the general result, by the previous case, if suffices to show that $\Ker(f) = \Ker(g)$ and by symmetry, it is enough to show that $\Ker(f) \subseteq \Ker(g)$. 
Both $f$ and $g$ factor through a torus $T'$ with $(BT') \pcom \simeq BT$. We then have
\[ Bj Bf' \simeq Bj Bg', \]
where $Bj \colon BT' \to BX$ is the composition of the $p$-completion map $BT' \to BT$ and $Bi$. If $a \in \Ker(f)$, then $Bj Bf' \kappa_a$ is nullhomotopic, 
hence so is $Bj Bg' \kappa_a = Bj \kappa_{g'(a)} $. Since $Bj$ is a monomorphism, $g'(a)=1$ and so $a \in \Ker(g')$. Thus $a \in \Ker(g)$. 
\end{proof}

The next result reduces the determination of the homotopy classes to a question regarding finite reflection groups over $\Z \pcom$. Recall that a finite reflection
group over a principal ideal domain $R$ is a finite subgroup $W$ of $\GL(L)$ generated by reflections, where $L$ is a finitely generated free $R$-module and a reflection
is a nontrivial element that fixes an $R$-submodule of corank one. Reflections do not necessarily have order two in this general context, so they are sometimes called
pseudo-reflections.

\begin{corollary}
\label{BijectionWithLattice}
If $X$ is a connected $p$-compact group with associated
$\Z \pcom$--reflection group $(W,L)$, then there is a bijection
\[ \frac{L/p^k L}{W} \to [B\Z/p^k,BX]. \]
\end{corollary}

\begin{proof}
We have bijections
\[ [B\Z/p^k,BX] \cong [B\Z/p^k,BT]/W \cong H^2(\Z/p^k;L)/W \cong \frac{L/p^k L}{W} \]
coming from Proposition \ref{prop-fromXtoTW}, the fact that $BT$ is a $K(L,2)$ and the naturality
in $M$ of the isomorphism $H^2(\Z/p^k;M) \cong M/p^k M$.  
\end{proof}

\begin{lemma}
\label{QuotientAndDual}
Let $W$ be a finite group and $A$ a finite abelian $p$-group with an action of $W$ by group
automorphisms. Then there is a bijection between $A/W$ and $\Hom(A,\Z/p^{\infty})/W$.
\end{lemma}

\begin{proof}
It is well known that $B^*:=\Hom(B,\Z/p^{\infty})$ is isomorphic to $B$ for any finite abelian $p$-group $B$. 
By Burnside's counting formula, it suffices to show that the cardinalities of $A^g=\Ker(g-1)$ and $(A^*)^{g^*}=\Ker(g^*-1)$ 
coincide for all $g \in W$. Note that the functor $\Hom(-,\Z/p^{\infty})$ is exact in the category of finite
abelian $p$-groups, since $\Z/p^{\infty}$ is $p$-divisible. Hence from the exact sequence
\[ 0 \longrightarrow \Ker(g-1) \longrightarrow A \stackrel{g-1}{\longrightarrow} A \longrightarrow \Coker(g-1) \longrightarrow 0 \]
we obtain the exact sequence
\[ 0 \longrightarrow \Coker(g-1)^* \longrightarrow A^* \stackrel{g^*-1}{\longrightarrow} A^* \longrightarrow \Ker(g-1)^* \longrightarrow 0. \]
Therefore
\[ \Ker(g^*-1) \cong \Coker(g-1)^* \cong \Coker(g-1) \]
and $\Coker(g-1)$ and $\Ker(g-1)$ have the same cardinality from the first exact sequence.
\end{proof}

\begin{corollary}
If $X$ is a connected $p$-compact group with associated
$\Z \pcom$--reflection group $(W,L)$, then there is a bijection
\[ \frac{L^*/p^k L^*}{W} \to [B\Z/p^k,BX], \]
where $L^* = \Hom(L,\Z \pcom)$.
\end{corollary}

\begin{proof}
If we apply $\Hom(-,\Z \pcom)$ to the short exact sequence
\[ 0 \to L \stackrel{p^k}{\longrightarrow} L \to L/p^k L \to 0, \]
we obtain an exact sequence
\[ 0 \to L^* \stackrel{p^k}{\longrightarrow} L^* \to \Ext(L/p^k L,\Z \pcom) \to 0, \]
because $\Hom(L/p^kL,\Z \pcom)=0$  and $\Ext(\Z \pcom,\Z \pcom)$ is torsion-free. Therefore 
$\Ext(L/p^k L,\Z \pcom) \cong L^*/p^k L^*$ as $W$--modules. The short exact sequence
\[ 0 \to \Z \pcom \to \Q \pcom \to \Z/p^{\infty} \to 0 \]
gives us an isomorphism
\[ \Ext(L/p^k L,\Z \pcom) \cong \Hom(L/p^k L,\Z/p^{\infty}) \]
of $W$--modules. Therefore
\[ \left| \frac{L^*/p^k L^*}{W} \right| = \left| \frac{\Hom(L/p^k L,\Z/p^{\infty})}{W} \right| = \left| \frac{L/p^k L}{W} \right|, \]
where the last equality follows from Lemma \ref{QuotientAndDual}. The result follows from Corollary \ref{BijectionWithLattice}.
\end{proof}

The previous corollary could have been proved using the fact that $(W,L)$ and $(W,L^*)$ are isomorphic as $\Z \pcom$-reflection groups,
but the proof given here is more elementary. Note that 
\[ L^* = \Hom(L,\Z \pcom) = \Hom(\pi_2(BT),\Z \pcom) \cong H^2(BT;\Z \pcom), \]
hence the action of $W$ on $H^2(BT;\Z \pcom)$ can also be used to determine the size of $[B\Z/p^n,BX]$.

\begin{corollary}
Let $\hat{T}$ be a discrete approximation to the maximal torus $T$ of the connected $p$-compact group $X$. For any cyclic $p$-group $A$ there
is a bijection
\[ \Hom(A,\hat{T})/W_X(T) \to [BA,BX] \]
\end{corollary}

\begin{proof}
By Proposition \ref{prop-fromXtoTW}, there is a bijection between $[BA,BX]$ and $[BA,BT]/W_X(T)$. The result follows
from the $W_X(T)$-equivariant bijections
\[ [BA,BT] \cong [BA,B\hat{T}] \cong \Hom(A,\hat{T}) \qedhere \]
\end{proof}

For an abelian group $A$, let us denote by $\Omega_m(A)$ the subgroup of elements of $A$
of order dividing $m$.

\begin{corollary}
\label{CalculoConAproximacionDiscreta}
If $X$ is a connected $p$-compact group and $\hat{T}$ is a discrete approximation to its maximal torus $T$, then
there is a bijection
\[ \Omega_{p^k}(\hat{T})/W_X(T) \to [B\Z/p^k,BX] \]
\end{corollary}

The results above can also be generalized to $p$-local compact groups with a connectivity
condition. Recall that $p$-local compact group is a triple $(S,\Ff,\Ll)$, where $S$ is a
discrete $p$-toral group, $\Ff$ is a saturated fusion system over $S$ and $\Ll$ is a centric
linking system associated to $\Ff$.

\begin{proposition}
Let $(S,\Ff,\Ll)$ be a $p$-local compact group and let $\hat{T}$ be the subgroup of $S$ of 
infinitely $p$-divisible elements. If any element of $S$ is $\Ff$--conjugate to an element
of $\hat{T}$, then for any cyclic $p$-group $A$ there is a bijection
\[ \Hom(A,\hat{T})/\Aut_{\Ff}(\hat{T})  \to [BA,|\Ll| \pcom].  \]
\end{proposition}

\begin{proof}
Let $\theta \colon BS \to |\Ll| \pcom$ be the natural inclusion followed by completion. 
By \cite[Theorem 6.3(a)]{BLO2}, the map  
\begin{align*}
\Rep(A,\Ll) & \to [BA,|\Ll| \pcom], \\
 [h] & \mapsto [\theta \circ Bh], 
\end{align*}
is a bijection. Recall that $\Rep(A,\Ll) = \Hom(A,S)/{\sim}$, where
two homomorphisms $f_1$, $f_2 \colon A \to S$ are related if there exists
$\chi \in \Hom_{\Ff}(f_1(A),f_2(A))$ such that $f_2 = f_1 \circ \chi$. Let
$j$ denote the inclusion of $\hat{T}$ in $S$. We will show that the map
\begin{align*}
\Hom(A,\hat{T})/\Aut_{\Ff}(\hat{T}) & \to \Rep(A,\Ll), \\
[h] & \mapsto [jh],
\end{align*}
is a bijection. If $[jh_1]=[jh_2]$, then there exists $\chi \in \Hom_{\Ff}(jh_1(A),jh_2(A))$ such that 
$jh_2 = \chi \circ jh_1$. By \cite[Lemma 2.4(b)]{BLO2}, the map $\chi$ extends to an element $\omega \in \Aut_{\Ff}(\hat{T})$
and therefore $[h_1]=[h_2]$.

Given $ g \colon A \to S$ and a generator $a$ of $A$, there exists $s \in S$ such that 
$sg(a)s^{-1} \in \hat{T}$. Then $[c_s g]$ belongs to the image, and this shows surjectivity
since $[c_s g]=[g]$ in $\Rep(A,\Ll)$.
\end{proof}

By \cite[Proposition 10.5 and Theorem 10.7]{BLO2}, for each connected $p$-compact group $X$, there exists a $p$-local compact group $(S,\Ff_X,\Ll_X)$
such that $|\Ll_X| \pcom \simeq BX$. More explicitly, there exists a discrete approximation $S$ 
of $N_p(T)$ such that $\hat{T}$ is a discrete approximation of $T$, and the morphisms in $\Ff_X$ are given by 
\[ \Hom_{\Ff_X}(P,Q)=\{ \varphi \in \Hom(P,Q) \mid \theta_{|BQ} B\varphi \simeq \theta_{|BP} \} \]
In particular, $\Aut_{\Ff_X}(\hat{T})$ is isomorphic to $W_X(T)$. The
argument for surjectivity in the proof of Proposition \ref{prop-fromXtoTW} can be adjusted to show that any element
of $S$ is $\Ff_X$--conjugate to an element of $\hat{T}$.

\begin{remark}
The condition that any element of $S$ is $\Ff$--conjugate to an element of $\hat{T}$ is part
of the tentative definition of connected $p$-local compact group in \cite[Definition 3.1.4]{Gon},
which was discarded later by the same author for the more precise notion of irreducibility in
\cite[Definition 3.1]{Gon2}.
\end{remark}

\section{The computation in the non-modular cases}
\label{NonModularCases}

In this section we determine a formula for the cardinality of $[B\Z/p^n,BX]$, for any
non-modular connected $p$-compact group $X$, which is given in terms of the exponents
of the associated $\Z \pcom$-reflection group.

Given a principal ideal domain $R$, let us recall that an $R$-root datum is a triple 
\[\D = (W, L, \{Rb_\sigma \mid \sigma \in J\}),\]
where $L$ is a finitely generated free $R$-module, $W$ is a finite subgroup of $\Aut_R(L)$ generated by reflections, and $J$
is the set of reflections of $W$. Each $b_\sigma$ is related to a generating reflection $\sigma\in W$ via the formula $\sigma(x) = x-\beta_\sigma(x)b_\sigma$, 
where $\beta_\sigma\colon L\to R$ is $R$-linear, and $g(Rb_\sigma) = Rb_{g\sigma g^{-1}}$ for all $g\in W$. Note that in this context a reflection
is a nontrivial element that fixes an $R$-submodule of corank one, but it does not necessarily have order two. The element $b_\sigma\in R$ is the coroot 
associated to $\sigma$ and dually, the map $\beta_\sigma$ is the root associated to $\sigma$. We will often just write $\D = (W, L)$.

Crystallographic root systems, which give rise to compact connected Lie groups, correspond to $\Z$-root data. The fundamental group of a compact connected Lie group 
$G$ is isomorphic to 
\[\pi_1(G)\cong P/Q,\]
where $Q$ is the $\Z$-lattice generated by a fundamental root system and $P$ is the $\Z$-lattice of their associated weights. Translating $P$ and $Q$ to their associated $\Z$-root datum gives $P=L^*$ and $Q=L_0^*$, where $L_0 = \spann_\Z(\{b_\sigma\})$. In general we may define $L_0$ for an $R$-root datum $\D$ as $\spann_R(\{b_\sigma\})$, and the fundamental group of $\D$ is then defined as
\[\pi_1(\D) := L/L_0.\]

Specializing to $\Z\pcom$-root data, for each connected $p$-compact group $X$, we have by \cite[Theorem 1.1]{DW3} an isomorphism
\[\pi_1(\D) \cong \pi_1(X),\]
where $\D$ is the $\Z\pcom$-root datum corresponding to $X$ under the classification of connected $p$-compact groups. Now the classification 
of $\Z\pcom$-root data \cite[Theorem 8.1]{AG} states that $\D = \D_1\times \D_2$, where $\D_1= \D^\prime\otimes_\Z\Z\pcom$ for a $\Z$-root 
datum $\D^\prime=(W_1,L^\prime)$, and $\D_2 = (W_2,L_2)$ is an exotic $\Z\pcom$-root datum. Exotic $\Z\pcom$-root data have trivial fundamental group, 
so we obtain that
 \[\pi_1(\D) = \pi_1(\D^\prime)\otimes_\Z\Z\pcom,\] 
hence the torsion subgroup of $\pi_1(\D)$ is precisely the $p$-Sylow subgroup of $\pi_1(\D^\prime)$. We record the following statement for future computations.

\begin{lemma}\label{non-modular-pi1-torsionfree}
Let  $\D=(W,L)$ be a $\Z\pcom$-root datum. If $p$ and $|W|$ are relatively prime, then $\pi_1(\D)$ is torsion-free. 
\end{lemma}

\begin{proof}
This follows from the fact that for compact connected semisimple Lie groups the connection index $|P/Q|$ divides the order of the Weyl group, see for example \cite[Theorem 11-6]{Kane}.
\end{proof}

The following lemma is essentially the same idea as the proof of \cite[Proposition 8.2-i)]{Springer} for crystallographic root systems.

\begin{lemma}\label{coker1-gsss}
Let $\D=(W,L)$ be a $\Z\pcom$-root datum and let $g\in W$. Then there is a short exact sequence
\[0\to L_0/{\im}(1-g)\to\mathrm{Coker}(1-g)\to \pi_1(\D)\to 0. \]
\end{lemma}

\begin{proof}
We only need to show that the image of $1-g$ is contained in $L_0$. By assumption $\im(1-\sigma)\subset \Z\pcom b_\sigma$, for every reflection $\sigma\in W$. Writing $g=\sigma_1\cdots\sigma_h$ 
as a product of reflections, we have 
\[1-g=\sigma_1(1-g^\prime)+1-\sigma_1,\]
where $g^\prime=\sigma_2\cdots\sigma_h$. Since $\sigma(L_0) = L_0$ for every reflection $\sigma$, inductively we obtain that $\im(1-g)\subset L_0$. 
\end{proof}

The next result is a non-modular version of \cite[Corollary 8.3]{Springer}. Recall that a reflection subgroup is a subgroup generated
by reflections.

\begin{proposition}\label{det(1-g)&p-relprime}
Let $\D=(W,L)$ be a $\Z\pcom$-root datum and let $g\in W$. If $W$ is irreducible, non-modular and no proper reflection subgroup contains $g$, then ${1-g}$ is invertible. 
\end{proposition}

We first need to lay out some facts before proving this result. Let us consider the map  $\GL(L)\to \GL(L/pL)$ induced by the projection $L\to L/pL$. When $p > 2$, which always holds 
in the non-modular case, the composite $W\hookrightarrow\GL(L)\to \GL(L/pL)$ is injective (see \cite[Lemma 11.3]{AGMV}), hence $W$ is a reflection group over $\F_p$. We will need the following version of Steinberg's
fixed point theorem.
 
 \begin{lemma}
\label{lema1}
Let $V$ be a finite-dimensional vector space over $\F_p$ and let $G\subset \GL(V)$ be a non-modular finite reflection group. Then the isotropy group $G_\Gamma$ of any subset $\Gamma\subset V$
is a reflection subgroup.
 \end{lemma}
  
 \begin{proof}
Since $G$ is non-modular, the ring $\F_p[V]^{G}$ is polynomial by \cite[Theorem 18-1]{Kane}. Then a result of Nakajima (see \cite[Corollary 1.3]{Smith}) shows 
that $\F_p[V]^{G_\Gamma}$ is a polynomial algebra. The lemma follows from a well-known theorem by Serre \cite{Serre}.
 \end{proof}

\begin{proof}[Proof of Proposition \ref{det(1-g)&p-relprime}]
We will show that $\Coker(1-g)$ is trivial. First we claim that $(1-g)\otimes\Q \pcom$ is invertible. If not, then we may find a vector $v\in V=L \otimes_{\Z \pcom} {\Q \pcom}$ such that $g$ fixes $v$. By \cite[Proposition 26-6]{Kane}, the stabilizer $G_v\subset W$ is a reflection subgroup. Our assumption on $W$ forces $G_v = W$, but this is impossible since $W$ is irreducible. It follows that $\Coker(1-g)$ is a torsion group.

Now let us show that $L_0=\im(1-g)$. If $\im(1-g)$ were a proper sub-lattice of $L_0$, there would exist $x\in L$ such that $(1-g)x\in pL_0$ and $x \notin pL$. Such $x\in L$ would become a non-trivial fixed point in $L/pL$ under the action of the element $g$, and thus by Lemma \ref{lema1}, the stabilizer of $x+pL$ would be a reflection group over $\F_p$. Up to conjugation we may further lift the stabilizer to a $\Z\pcom$-reflection subgroup of $W$. The same reasoning as in our first claim shows that this is not possible under our assumptions. Lemma \ref{coker1-gsss} then implies that $\Coker(1-g)= \pi_1(\D)$, but from Lemma \ref{non-modular-pi1-torsionfree} we have that $\pi_1(\D)$ is torsion-free, hence trivial.
\end{proof}

For the convenience of the reader we now outline how the proof of \cite[Theorem 3]{Dj1} adapts to an arbitrary $\Z\pcom$-root datum $\D=(W,L)$.

\begin{corollary}
\label{NotModular}
If $X$ is a non-modular connected $p$-compact group, then
\[ |[B\Z/p^k,BX]| = \prod_{i=1}^l \frac{m_i+p^k}{m_i+1} \]
where $m_i$ are the exponents of $W_X(T)$ regarded as a
reflection group over $\Z \pcom$.
\end{corollary}

\begin{proof}
Let $\D=(W,L)$ be a $\Z\pcom$-root datum. Let $g\in W$ and let $\D_1=(W_1,L_1)$ be a minimal sub-$\Z\pcom$-root datum of $\D$ such that $g\in W_1$. We may factor $\D_1$ into irreducible root data so that $W_1 = W_{11} \times \cdots \times W_{1r}$, where each $W_{1i}$ is an irreducible reflection subgroup over $\Z\pcom$. Then we have that $g = g_1\cdots g_r$, the component-wise representation of $g$, and each $W_{1i}$ is a minimal reflection subgroup containing $g_i$. If $s$ be the multiplicity of the eigenvalue $1$ of $w$, then the rank of $L_1$ equals $l-s$ and by \cite[Theorem III.12]{N}, which also holds over any principal ideal domain, $g$ may be written as a matrix over $\Z\pcom$ in the form  \[\left(\begin{matrix}
A &B\\
0&C
\end{matrix}\right),\] 
where $A$ is an upper triangular $s \times s$ matrix with ones on its diagonal. Since $g$ has finite order, $A=I_s$. As $\im C=\im g|_{L_1}$, Proposition \ref{det(1-g)&p-relprime} implies that $C-I$, regarded over $\Z\pcom/{p^k}\Z\pcom$ for any $k\geq1$, is a block sum of invertible matrices. Consequently the number of elements of $L/p^kL$ fixed by $g$ equals to $({p^k})^s$.

Let $h_i$ be the number of elements of $W$ with an invariant subspace of $L/p^kL$ of dimension $i$. The non-modular version of a result of Solomon \cite[Theorem A 31-1]{Kane} states that 
\[\prod_{i=1}^l(t+m_i) = h_0+h_1t+\cdots+h_lt^l,\]
where the $m_i$'s are the exponents of $W$.  The previous two items and the Burnside counting formula yield the desired formula.
\end{proof}

\begin{remark}
\label{NonModularElements}
The above matrix expression of $g$ actually gives that \[\Coker(g-1)= L/L_1^\prime\oplus \Coker(g|_{L_1}-1).\] Thus, the torsion subgroup of $\Coker(g-1)$ is the same as the torsion subgroup of $\Coker(g|_{L_1}-1)$. In virtue of Proposition \ref{det(1-g)&p-relprime}, for any reflection group $G$ and an element $g\in G$, we may conclude that if $H\subset G$ is a non-modular subgroup such that $g\in H$, then $\Coker(1-g)$ is torsion free.
\end{remark}

\section{Generalized Grassmannians}
\label{generalizedgrassmanians}

In this section we focus on the irreducible $p$-compact groups called 
generalized Grassmannians, more particularly in the family 2a. 

Generalized Grassmanians are parametrized by triples $(m,s,n)$ of positive integers 
with $s | m $ which satisfy certain conditions depending on the prime $p$. The $p$-compact group $X(m,s,n)$
has rank $n$ and its Weyl group is $G(m,s,n)$, the group of monomial $n \times n$ matrices whose non-zero 
entries are $m$th roots of unity and whose determinant is an $(m/s)$th root of unity. Equivalently, 
it is the semidirect product of the groups 
\[ A(m,s,n) = \{ (x_1,\ldots,x_n) \in (\Z/m)^n \mid x_1 + \cdots + x_n \equiv 0 \Mod s \} \]
and $\Sigma_n$, with the permutation action. 

Generalized Grassmannians are usually split in four families. Since compact Lie groups were already
covered in \cite{DP}, \cite{Dj1}, \cite{Dj2} and \cite{PW}, we ignore the generalized Grassmanians in family 1, $X(2,s,n)$ in family 2a,
$X(3,3,2)$, $X(4,4,2)$ and $X(6,6,2)$ in family 2b and $X(2,1,1)$ in family 3. The rest of $p$-compact
groups $X(m,m,2)$ in family 2b are non-modular, since the order of $G(m,m,2)$ is $2m$ and $p \equiv \pm 1 \Mod m$
when $m \neq 3,4,6$. So are the rest of $p$-compact groups $X(m,1,1)$ in family 3 since the order of $G(m,1,1)$
is $m$ and $p \equiv 1 \Mod m$ when $m >2$. Hence Corollary \ref{NotModular} can be used for them.

From now on, we focus on generalized Grassmanians $X(m,s,n)$ in the family 2a with $m>2$. Note that $n \geq 2$, $ m \neq s$ if $n=2$ 
and $ p \equiv 1 \Mod m$, in particular, $p \neq 2$. Since $m$ divides $p-1$ and $\Z/(p-1)$ is a
subgroup of the units of $\Z \pcom$, we can regard $\Z/m$ as a subgroup of the group
of units of $\Z \pcom$. To be more precise, let $a$ be a primitive $(p-1)$-th root
of unity in $\Z \pcom$ and let $ b = a^{p-1/m}$. Then the action of $G(m,s,n)$ on the
discrete approximation $(\Z/p^{\infty})^n$ of its maximal torus is given by
\[ (r_1,\ldots,r_n,\sigma) (y_1,\ldots,y_n) = (b^{r_1} y_{\sigma^{-1}(1)}, \ldots, b^{r_n} y_{\sigma^{-1}(n)} ). \]
In order to use Corollary \ref{CalculoConAproximacionDiscreta}, we will find a fundamental domain for the action of $G(m,s,n)$ on $\Omega_{p^k}((\Z/p^{\infty})^n) \cong (\Z/p^k)^n$.
Let $c$ be the residue mod $p^k$ of $b$. Since $b$ is a unit in $\Z \pcom$, so is $c$ in $\Z/p^k$ and we can consider 
the multiplicative subgroup $H$ of $(\Z/p^k)^{\times}$ generated by $c$. Let $K$ be the subgroup generated by $c^s$.


The action of $H$ breaks $\Z/p^k$ into $(p^k-1)/m$ orbits $C_0, C_1, \ldots, C_{(p^k-1)/m}$,
where $C_0$ is the orbit of the zero element. Note that each orbit $C_j$ with $j \neq 0$
has $m$ elements and the action of $K$ breaks each one of them into $s$ orbits. Given
$z \in S \subseteq \Z/p^k-\{[0]\}$, we will say that $z$ is the minimum of $S$ if $z=[i]$, where
$i$ is the minimum of the set
\[ \{ j \mid 1 \leq j < p^k, [j] \in S \}. \]
If $S=\{[0]\}$, we say that $[0]$ is the minimum of $S$.

\begin{definition}
We say that the element $(y_1,\ldots,y_n) \in C_{i_1} \times \ldots \times C_{i_n} \subseteq (\Z/p^k)^n$ is distinguished if
the following three conditions are satisfied.
\begin{enumerate}
\item $i_1 \leq \ldots \leq i_n$.
\item If $j \leq n-1$, then $y_j$ is the minimum of its $H$-orbit $C_{i_j}$.
\item The element $y_n$ is the minimum of its $K$-orbit.
\end{enumerate}
\end{definition}

It is clear that any element in $(\Z/p^k)^n$ is in the $G(m,s,n)$-orbit of a distinguished element.

\begin{lemma}
\label{FundamentalDomain}
The set of distinguished elements is a fundamental domain for the action of $G(m,s,n)$ on $(\Z/p^k)^n$. 
\end{lemma}

\begin{proof}
Let $x=(x_1,\ldots,x_n)$ and $y=(y_1,\ldots,y_n)$ be two distinguished elements in the same $G(m,s,n)$-orbit,
that is, $x=(r_1,\ldots,r_n,\sigma)y$ for some $(r_1,\ldots,r_n,\sigma) \in G(m,s,n)$. Since the action of 
$G(m,s,n)$ is given by permuting elements and multiplying by powers of $c$, the number of coordinates of $x$ 
and $y$ that belong to a given $H$-orbit coincide. 

Assume first that all the coordinates of $x$ and $y$ belong to the same $H$-orbit. If $a$ is the minimum of this
$H$-orbit, then $x=(a,\ldots,a,b)$ and $y=(a,\ldots,a,d)$. If $\sigma$ fixes $n$, then 
\[ (a,\ldots,a,b) = (c^{r_1}a,\ldots,c^{r_{n-1}}a,c^{r_n}d), \]
from where $r_j$ is a multiple of $m$ for each $j \leq n-1$, in particular a multiple of $s$. Since $r_1+\cdots+r_n$
is a multiple of $s$, so must be $r_n$. But then $b$ and $d$ lie in the same $K$-orbit. Since they are both the minima
of their $K$-orbits, we obtain $b=d$.

If $\sigma$ does not fix $n$, then we will have
\begin{align*}
a & = c^{r_j}d, \\
b & = c^{r_i}a,
\end{align*}
for certain $i$, $j$, and $a=c^{r_l}a$ in the rest of the coordinates. We obtain that $r_l$ is a multiple of $m$ if $l\neq i,j$
and therefore $r_i+r_j$ is a multiple of $s$. Since $b=c^{r_i+r_j}d$, the $K$-orbits of $b$ and $d$ are the same and because 
they are both the minima of their $K$-orbits, we obtain $b=d$.

Now assume that not all the coordinates of $x$ and $y$ belong to the same $H$-orbit. Then we have
\begin{align*}
x & = (a_1,\ldots,a_1,a_2,\ldots,a_2,\ldots,a_j,\ldots,a_j,b), \\
y & = (a_1,\ldots,a_1,a_2,\ldots,a_2,\ldots,a_j,\ldots,a_j,d).
\end{align*}
Note that $\sigma$ must preserve the blocks with equal coordinates $a_i$ for $i<j$ and the corresponding powers of $c$ in each of these blocks
must be trivial. But then the sum of the exponents of the remaining powers of $c$ must be a multiple of $s$ and therefore $(a_j,\ldots,a_j,b)$
and $(a_j,\ldots,a_j,d)$ would be in the same $G(m,s,n')$-orbit for some $n'<n$. By the previous case, we have $b=d$ and so $x=y$. 
\end{proof}

\begin{proposition}
Let $X(m,s,n)$ be a generalized Grassmannian in the family 2a with $m \geq 3$. The cardinality of the set
$[B\Z/p^k,BX(m,s,n)]$ equals
\[ 1 + \frac{p^k-1}{m} s + \sum_{j=1}^{p^k-1/m} {{n-2+j} \choose {j}} \left( \frac{p^k-1}{m} - j + 1 \right) s. \]
\end{proposition}

\begin{proof}
By Lemma \ref{FundamentalDomain} and Corollary \ref{CalculoConAproximacionDiscreta}, it suffices to compute the 
cardinality of set of distinguished elements of $(\Z/p^k)^m$. Now a distinguished element is given
by a sequence 
\[ (a_0,a_0,\ldots,a_0,a_1,\ldots,a_1, \ldots, a_j, \ldots, a_j, b), \]
where $a_i$ is the minimum of the set $C_i$, the element $b$ is the minimum in its $K$-orbit lying inside 
$C_i$ for some $i \geq j$ and $j \leq (p^k-1)/m$. 

Assume the element is not of the form $(0,\ldots,0,b)$. To count the set of distinguished elements
for a fixed $j$, we only need to count how many times each $a_i$ repeats and the possible values of
$b$. For the first part, we are counting sequences $(n_0,\ldots,n_j)$ of nonnegative integers with
$ n_0 + \cdots + n_j = n - 1$ and $n_j \geq 1$. Equivalently, sequences $(n_0,\ldots,n_j)$ of nonnegative
integers with $n_0 + \cdots + n_j = n - 2$. These sequences are weak $(j+1)$-compositions of $n-2$
and the number of such sequences is given by
\[ {{n-2+j} \choose {j}}. \]
The element $b$ lies in the set of minima of $K$-orbits of elements of $C_i$ with $i \geq j$, which
has cardinality
\[ \left( \frac{p^k-1}{m} - j + 1 \right) s. \]
Therefore the cardinality of $[B\Z/p^k,X(m,s,n)]$ is given by
\[ 1 + \frac{p^k-1}{m} s + \sum_{j=1}^{p^k-1/m} {{n-2+j} \choose {j}} \left( \frac{p^k-1}{m} - j + 1 \right) s, \]
as we wanted to prove.
\end{proof}

The argument above applies to the groups $G(m,1,1)$ in the family 3 as long as $p \neq 2$, obtaining the formula
\[ | [B\Z/p^k,BX(m,1,1)] | = 1 + \frac{p^k-1}{m}. \]
The $p$-compact groups $X(m,1,1)$ are non-modular if $m>2$, or if $m=2$ and $p \neq 2$, hence we could also use 
Corollary \ref{NotModular} in those cases and the result agrees since the exponent of $G(m,1,1)$ is $m-1$. Note 
that $X(m,1,1)$ is the Sullivan sphere $(S^{2m-1})\pcom$. 

\section{The rest of modular cases}
\label{ModularCases}

The remaining modular cases which do not correspond to compact Lie groups are $X_{12}$ at the prime $3$, $X_{24}$
at the prime $2$, $X_{29}$ and $X_{31}$ at the prime $5$ and $X_{34}$ at the prime $7$. In this section we treat
the first four in detail, while the computation for $X_{34}$ is achieved using GAP.

Since $[B\Z/p^n,BX]$ is in bijective correspondence with the set of $W_X(T)$-orbits in $L/p^n L$ by Corollary \ref{BijectionWithLattice}, 
we will use Burnside's counting formula
\[ | X/G | = \frac{1}{|G|} \sum_{g \in \cc(G)} |G/C_G(g)| \cdot | X^g |, \]
where $\cc(G)$ is a set of representatives of the conjugacy classes of $G$. Let $\rho \colon W_X(T) \to GL(L)$ be the
homomorphism that makes $W_X(T)$ a finite reflection group over $\Z \pcom$. The action on $L/p^n L$ corresponds
to the homomorphism $\rho_n$ given as the composition 
\[ W_X(T) \to GL(L) \to GL(L/p^n L),\] 
where the second map is mod $p$ reduction. Since the action is linear, we need to determine $\Ker(\rho_n(w)-I)$ 
for a representative $w$ of each conjugacy class in $W_X(T)$. The following result will help us identify elements 
without nontrivial fixed points.

\begin{lemma}
If $\Ker(\rho_n(w)-I)$ is nontrivial for some $n>1$, so is $\Ker(\rho_1(w)-I)$.
\end{lemma}

\begin{proof}
Let $v+p^n L$ be a nontrivial fixed point for $\rho_n(w)$. Reducing modulo $p$ we obtain
\[ \rho_1(w)(v+pL)=v+pL. \]
If $v \notin pL$, then $v+pL$ is a nontrivial fixed point for $\rho_1(w)$. If $v \in pL$, let $ 1 \leq k <n$ be the
maximum integer such that $ v \in p^k L$, so that $v=p^k v'$ with $v' \notin pL$. Then
\[ p^k \rho_n(w)(v'+p^n L) = \rho_n(w)(p^k v'+p^nL) = p^k v' + p^n L, \]
from where $\rho(w)(v')-v' \in p^{n-k} L$. Hence $\rho_1(w)(v'+pL)=v'+pL$ and $v'+pL$ is a nontrivial fixed point
for $\rho_1(w)$. 
\end{proof}

The next lemma will give us the number of fixed points.

\begin{lemma}
\label{PuntosFijosConRangoYTorsion}
Let $r$ be the rank of $\Ker(\rho(w)-I)$ over $\Z \pcom$ and let $A$ be the torsion submodule of $\Coker(\rho(w)-I)$. Then
$\Ker(\rho_n(w)-I)$ is an extension of $A/p^n A$ by $(\Z/p^n)^r$. In particular, the cardinality of $\Ker(\rho_n(w)-I)$ equals $p^{nr} |A/p^nA|$.
\end{lemma}

\begin{proof}
Since $L$ is a free $\Z \pcom$--module and $\Z \pcom$ is a principal ideal domain, we have an exact sequence 
\[ 0 \longrightarrow (\Z \pcom)^r \longrightarrow L \stackrel{\rho(w)-I}{\longrightarrow} L \longrightarrow F \oplus A \longrightarrow 0, \]
where $F$ is a free $\Z \pcom$--module and $A$ is a torsion $\Z \pcom$--module. We break this exact sequence into two short exact sequences
\begin{align*}
0 & \longrightarrow F' \stackrel{h}{\longrightarrow} L \longrightarrow F \oplus A \longrightarrow 0, \\
0 & \longrightarrow (\Z \pcom)^r \longrightarrow L \stackrel{g}{\longrightarrow} F' \longrightarrow 0 ,
\end{align*}
such that $\rho(w)-I = hg$. Note that 
\[ \rho_n(w) - I = (\rho(w)-I) \otimes 1_{\Z/p^n} = (h \otimes 1_{\Z/p^n})(g \otimes 1_{\Z/p^n}). \]
Since $g$ is surjective, so is $g \otimes 1_{\Z/p^n}$ and therefore there is a short exact sequence
\[ 0 \to \Ker(g \otimes 1_{\Z/p^n}) \to \Ker(\rho_n(w) -I) \to \Ker(h \otimes 1_{\Z/p^n}) \to 0. \]
Note that $F'$ is a $\Z \pcom$--submodule of $L$, hence it is free and therefore the sequence
\[ 0 \longrightarrow (\Z \pcom)^r \otimes_{\Z \pcom} \Z/p^n \longrightarrow L \otimes_{\Z \pcom} \Z/p^n \stackrel{g \otimes 1_{\Z/p^n}}{\longrightarrow} F' \otimes_{\Z \pcom} \Z/p^n \longrightarrow 0 \]
is exact. Thus $\Ker(g \otimes 1_{\Z/p^n}) \cong (\Z/p^n)^r$. On the other hand, we have an exact sequence
\[ 0 \to \Tor_1^{\Z \pcom}(F \oplus A,\Z/p^n) \to F' \otimes_{\Z \pcom} \Z/p^n \stackrel{h \otimes 1_{\Z/p^n}}{\longrightarrow} L \otimes_{\Z \pcom} \Z/p^n \to (F \oplus A) \otimes_{\Z \pcom} \Z/p^n \to 0, \]
from where
\[ \Ker(h \otimes 1_{\Z/p^n}) \cong \Tor_1^{\Z \pcom}(F \oplus A,\Z/p^n) \cong \Tor_1^{\Z \pcom}(A,\Z/p^n) \cong A/p^nA. \]
The desired result follows.
\end{proof}

For each $w \in W_X(T)$, let $r(w)$ be the rank of $\Ker(\rho(w)-I)$ over $\Z \pcom$ and $t_n(w) = |A_w/p^n A_w|$, where $A_w$ is the torsion
submodule of $\Coker(\rho(w)-I)$. 

\begin{corollary}
\label{FormulaCasiGeneral}
If $X$ is an exotic $p$-compact group, we have
\[ |[B\Z/p^n,BX]| = \frac{1}{|W_X(T)|} \sum_{w \in \cc(W_X(T))} \frac{|W_X(T)|}{|C_{W_X(T)}(w)|} p^{nr(w)} t_n(w). \]
\end{corollary}

We can improve this formula using the same idea from the proof of Corollary \ref{NotModular}.

\begin{corollary}
\label{FormulaGeneral}
Let $X$ be an exotic $p$-compact group and let $m_i$ be the exponents of $W_X(T)$ regarded as a
reflection group over $\Z \pcom$. If $R$ is a set of representatives of conjugacy classes
of elements $w \in W_X(T)$ such that $\Coker(\rho(w)-1)$ has nontrivial torsion, then
\[ |[B\Z/p^n,BX]| = \frac{1}{|W_X(T)|} \left( \prod_{i=1}^l (m_i+p^n) + \sum_{w \in R} \frac{|W_X(T)|}{|C_{W_X(T)}(w)|} p^{nr(w)}(t_n(w)-1) \right). \]
\end{corollary}

\begin{proof}
Note that $r(w)$ equals the dimension of the kernel of $(\rho(w)-I) \otimes_{\Z \pcom} \Q \pcom$. If we regard $W_X(T)$ as a finite reflection group
over $\Q \pcom$, then Solomon's formula (see \cite{S} or \cite[Theorem 9.3.4]{Sm}) gives us
\[  \prod_{i=1}^l (t + m_i) = h_0 + h_1t + \cdots + h_lt^l, \]
where $h_i$ is the number of elements of $W_X(T)$ with an invariant subspace of $L \otimes_{\Z \pcom} \Q \pcom$ of dimension $i$, hence
\[ \sum_{w \in \cc(W_X(T))} \frac{|W_X(T)|}{|C_{W_X(T)}(w)|} p^{nr(w)} = h_0 + h_1 p^n + \cdots + h_l p^{ln} = \prod_{i=1}^l (p^n + m_i). \]
The result follows from Corollary \ref{FormulaCasiGeneral}.
\end{proof}

\begin{remark}
The formula from Corollary \ref{FormulaGeneral} could also be expressed in the form
\[ |[B\Z/p^n,BX]| = \prod_{i=1}^l \frac{m_i+p^n}{m_i+1} + \sum_{w \in R} \frac{p^{nr(w)}(t_n(w)-1)}{|C_{W_X(T)}(w)|}, \]
which can be compared with Corollary \ref{NotModular}, exhibiting the difference with the non-modular case.
\end{remark}

Recall that Remark \ref{NonModularElements} showed that if $w$ belongs to a non-modular reflection subgroup, then $\Coker(\rho(w)-I)$ has no torsion.
Now we will use the following steps to compute $|[B\Z/p^n,BX]|$ for the connected $p$-compact group $X$ corresponding to the finite reflection group $\rho \colon W_X(T) \to GL_l(\Z \pcom)$.
\begin{enumerate}
\item Determine a set $\cc(W_X(T))$ of representatives of the conjugacy classes of $W_X(T)$.
\item Find as many non-modular reflection subgroups of $W_X(T)$ as possible. Remove from $\cc(W_X(T))$ those
elements which can be conjugated into these subgroups.
\item For the remaining elements, find out whether their mod $p$ reductions have nontrivial fixed points. 
\item For each remaining element $w$, determine the Smith normal form of $\rho(w)-I$ to find whether its cokernel has torsion. If so, recover
$r(w)$ and $t_n(w)$ from the normal form, and determine the size of the conjugacy class of $w$. 
\item Use Corollary \ref{FormulaGeneral} to compute $|[B\Z/p^n,BX]|$.
\end{enumerate}

When the group $W_X(T)$ has a small number of conjugacy classes, it may be easier to skip Step (2), compute the fixed points of 
the mod $p^n$ reduction of $g-1$ directly instead of Step (4) and use Burnside's counting formula.

\subsection{The $3$-compact group $X_{12}$}

An explicit description of $G_{12}$ as a finite complex reflection group can be found in pages 201-203
of \cite{Sm}. The image of the representation $G_{12} \to GL_2(\C)$ is generated by the matrices
\[ \left( \begin{array}{cc} 0 & 1 \\ -1 & 0 \end{array} \right), \qquad  \frac{1}{\sqrt{-2}} \left( \begin{array}{cc} -1 & 1 \\ 1 & 1 \end{array} \right), \qquad 
\left( \begin{array}{cc} \omega & 1/2 \\ -1/2 & \overline{\omega} \end{array} \right), \qquad \left( \begin{array}{cc} 0 & 1 \\ 1 & 0 \end{array} \right), \]
donde $\omega = \frac{-1+\sqrt{-2}}{2}$. This is in fact a representation over $\Q(\sqrt{-2})$, which
can be achieved over $\Z^{\wedge}_3$ by replacing $\omega$ with the solution of the equation $(2x+1)^2=-2$ 
which is a multiple of three, and $\overline{\omega}$ with the solution which is congruent to $2$ mod $3$.

As an abstract group, $G_{12} \cong GL_2(\F_3)$. The representation $ \rho \colon GL_2(\F_3) \to GL_2(\Z ^{\wedge}_3)$ 
is such that the composition with mod $3$ reduction $GL_2(\Z ^{\wedge}_3) \to GL_2(\F_3)$ is a group isomorphism. 
Hence the homomorphism $\im(\rho) \to GL_2(\F_3)$ given by reduction mod $3$ is an isomorphism. The group $GL_2(\F_3)$ 
has eight conjugacy classes, with representatives
\[  \left( \begin{array}{cc}
1 & 0 \\
0 & 1 \end{array} \right), \qquad
\left( \begin{array}{cc}
-1 & 0 \\
0 & -1 \end{array} \right), \qquad
\left( \begin{array}{cc}
0 & 1 \\
-1 & 0 \end{array} \right), \qquad
\left( \begin{array}{cc}
0 & 1 \\
1 & -1 \end{array} \right), \]
\[ \left( \begin{array}{cc}
0 & 1 \\
1 & 1 \end{array} \right), \qquad
a = \left( \begin{array}{cc}
1 & 0 \\
1 & 1 \end{array} \right), \qquad
\left( \begin{array}{cc}
-1 & 1 \\
0 & -1 \end{array} \right), \qquad
b = \left( \begin{array}{cc}
0 & 1 \\
1 & 0 \end{array} \right). \]
An element in $GL_2(\F_3)$ has nontrivial fixed points if and only if it is
the identity, or conjugate to $a$ or $b$. Following the steps from the beginning
of the section, we now search for elements in $\im(\rho)$ whose reduction
mod $3$ are $a$ and $b$. The matrix
\[ B = \left( \begin{array}{cc} 0 & 1 \\ 1 & 0 \end{array} \right) \]
clearly reduces to $b$ mod $3$. It is easy to check that the cokernel of $B-1$ is isomorphic to $\Z_3^{\wedge}$, hence
torsion-free. The matrix
\[ A = \left( \begin{array}{cc}
              -1/2 & \omega \\
              -\overline{\omega} & -1/2 \end{array} \right) = \left( \begin{array}{cc}
              \omega & 1/2 \\
              -1/2 & \overline{\omega} \end{array} \right) \left( \begin{array}{cc}
              0 & 1 \\
              -1 & 0 \end{array} \right)  \]
belongs to $\im(\rho)$ and reduces to $a$ mod $3$. The Smith normal form of $A-1$ is
\[ \left( \begin{array}{cc} 1 & 0 \\ 0 & 3 \end{array} \right), \]
hence its kernel is trivial and its cokernel is isomorphic to $\Z/3$. By Lemma \ref{PuntosFijosConRangoYTorsion},
the kernel of the mod $3^n$ reduction of $A-1$ is isomorphic to $\Z/3/3^n \Z/3 = \Z/3$. It is easy to check that
the centralizer of $a$ in $\GL_2(\F_3)$ has six elements, hence by Corollary \ref{FormulaGeneral},
\[ | [B\Z/3^n, BX_{12}] | = \frac{1}{48} \Big[  (5+3^n)(7+3^n) + 8 \cdot 2  \Big] = \frac{1}{48} ( 3^{2n} + 12 \cdot 3^n + 51 ). \]

\begin{remark}
In this computation, we see that since the cokernel of $A-1$ has torsion, we can not compute $| [B\Z/3^n,BX_{12}] |$ using Proposition \ref{NotModular}. 
Indeed, the formula obtained above differs from $(5+3^n)(7+3^n)/48$ by $1/3$.
\end{remark}

\subsection{The $2$-compact group $X_{24}$}

In \cite[Proposition 2.1]{OS}, there is a presentation of the group $G_{24}$ given by
\[ \langle s_1, s_2, s_3 \mid s_1^2, s_2^2, s_3^2, (s_2s_3s_2s_1)^4, (s_1s_3)^3, (s_2s_3)^3, (s_1s_2)^4 \rangle . \]
The representation $\rho \colon G_{24} \to \GL_3(\Z^{\wedge}_2) $ as a finite reflection group is defined by
\[ \rho(s_1) = \left( \begin{array}{ccc}
                       -1 & -\overline{\alpha} & 1 \\
                       0 & 1 & 0 \\
                       0 & 0 & 1 \end{array} \right), \quad \rho(s_2) = \left( \begin{array}{ccc}
                       1 & 0 & 0 \\
                       -\alpha & -1 & 1 \\
                       0 & 0 & 1 \end{array} \right), \quad \rho(s_3) = \left( \begin{array}{ccc}
                       1 & 0 & 0 \\
                       0 & 1 & 0 \\
                       1 & 1 & -1 \end{array} \right), \]
where $\alpha$ is the solution in $\Z^{\wedge}_2$ of $x^2-x+2=0$ such that $\alpha \equiv 3 \Mod 8$ and $\overline{\alpha}$
is the solution with $\overline{\alpha} \equiv 6 \Mod 8$. Let $a=\rho(s_1)$, $b=\rho(s_2)$ and $c=\rho(s_3)$. All reflection 
subgroups of $G_{24}$ are modular because all reflections in $G_{24}$ have order two and we are working at the prime two. Hence
in this case we skip Step (2).

As an abstract group, $G_{24}$ is isomorphic to $\Z/2 \times \GL_3(\F_2)$. If we let $a'$, $b'$ and $c'$ be the mod $2$ reductions
of $a$, $b$ and $c$, respectively, the representation $\Z/2 \times \GL_3(\F_2) \to \GL_3(\Z_2^{\wedge})$ sends $-I$ to $-I$ and $x'$
to $x$ for each $x \in \{a,b,c\}$. Conjugacy classes in $\GL_3(\F_2)$ and their sizes can be determined using the rational canonical 
forms. Elements of $\GL_3(\F_2)$ with characteristic polynomial $x^3+x^2+1$ or $x^3+x+1$ do not have nontrivial fixed points. The
rest of conjugacy classes are represented by $I$, $c'$, $a'c'$, $a'b'$, whose conjugacy class sizes are $1$, $21$, $56$ and $42$,
respectively. Therefore 
\[ \{ I,-I,c,-c,ac,-ac,ab,-ab \} \]
is a set of representatives of conjugacy classes in $G_{24}$ whose mod $2$ reductions have nontrivial fixed points. For each $x$ in
this set, it is easy to find the Smith normal form of $x-1$ and we summarize the result in Table \ref{table:conjugacyG24}.
\begin{table}[h]
\centering
\caption{}
\begin{tabular}{cc}
\toprule
Representative $x$ & Diagonal of Smith normal form of $x-I$ \\ \midrule
$I$ & $(0,0,0)$ \\ \midrule
$-I$ & $ (2,2,2)$ \\ \midrule
$c$ & $(1,0,0)$  \\ \midrule
$-c$ & $(1,2,0)$  \\ \midrule
$ac$ & $(1,1,0)$  \\ \midrule
$-ac$ & $(1,1,2)$  \\ \midrule
$ab$ & $(1,1,0)$  \\ \midrule
$-ab$ & $(1,1,4)$ \\ 
\bottomrule
\end{tabular}
\label{table:conjugacyG24}
\end{table}

We see that $-I$, $-c$, $-ac$ and $-ab$ are the only elements for which there is torsion in the cokernel. By Corollary \ref{FormulaGeneral}, the
cardinality of $[B\Z/2^n,BX_{24}]$ equals
\[  \frac{1}{336}\Big[ (3+2^n)(5+2^n)(13+2^n) + 1 \cdot 7 + 21 \cdot 2^n \cdot 1 + 56 \cdot 1 + 42 \cdot (2^{\min\{n,2\}}-1) \Big], \]
that is,
\[ \frac{1}{336}(2^{3n} + 21 \cdot 2^{2n} + 140 \cdot 2^n + 216 + 42 \cdot 2^{\min\{n,2\}}). \]
If $n \geq 2$, this expression can be simplified to 
\[ |[B\Z/2^n,BX_{24}]| = \frac{1}{336}(2^{3n} + 21 \cdot 2^{2n} + 140 \cdot 2^n + 384), \]
while 
\[ |[B\Z/2,BX_{24}]| = 2. \]

\begin{remark}
By \cite[Theorem 0.4]{L}, applying mod $2$ cohomology induces a bijection 
\[ [B\Z/2,BX_{24}] \cong \Hom_{\A}(H^*(BX_{24};\F_2),H^*(B\Z/2;\F_2)), \]
where $\A$ is the mod $2$ Steenrod algebra. It is well known that 
\[ H^*(B\Z/2;\F_2) \cong \F_2[x_1], \]
with $\Sq^1 x_1 = x_1^2$, while (see \cite[Page 213]{San} and \cite{SS})
\[ H^*(BX_{24};\F_2) \cong \F_2[c_8,c_{12},c_{14},c_{15}], \]
with 
\begin{align*}
\Sq^4 c_8 & = c_{12}, \\
\Sq^2 c_{12} & = c_{14}, \\
\Sq^1 c_{14} & = c_{15}, \\
\Sq^8 c_i & = c_8 c_i. 
\end{align*}
Any morphism $H^*(BX_{24};\F_2) \to H^*(B\Z/2;\F_2)$ of $\A$--algebras is determined by the
image of $c_8$, which can only be $x_1^8$ or $0$. Both options give morphisms of $\A$--algebras,
hence we recover $|[B\Z/2,BX_{24}]|=2$.
\end{remark}

\subsection{The $5$-compact group $X_{29}$}

We follow the description of $G_{29}$ in Section 8 of \cite{Ben}. The finite $5$-adic reflection group $G_{29}$ is generated by the four reflections
\[ r_1 = \frac{1}{2} \left( \begin{array}{cccc}
                            1 & -1 & -1 & -1 \\
                            -1 & 1 & -1 & -1 \\
                            -1 & -1 & 1 & -1 \\
                            -1 & -1 & -1 & 1 \end{array} \right),  \qquad r_2 = \left( \begin{array}{cccc}
                            0 & -\omega & 0 & 0 \\
                            \omega & 0 & 0 & 0 \\
                            0 & 0 & 1 & 0 \\
                            0 & 0 & 0 & 1 \end{array} \right), \]
\[ r_3 = \left( \begin{array}{cccc}
                            0 & 1 & 0 & 0 \\
                            1 & 0 & 0 & 0 \\
                            0 & 0 & 1 & 0 \\
                            0 & 0 & 0 & 1 \end{array} \right), \qquad r_4 = \left( \begin{array}{cccc}
                            1 & 0 & 0 & 0 \\
                            0 & 0 & 1 & 0 \\
                            0 & 1 & 0 & 0 \\
                            0 & 0 & 0 & 1 \end{array} \right), \]
where $\omega$ is a fourth root of unity in $\Z^{\wedge}_5$ with $\omega \equiv 2 \Mod 5$. Its center has order four and is generated
by $z= (r_1r_2r_3r_4)^5 = \omega I$. The normal subgroup $N$ generated by $(r_2r_3)^2$ and $z$ has order $64$. Equivalently, it is the
subgroup generated by the set
\[ R = \{ r_4(r_2r_3)^2r_4, r_1r_4(r_2r_3)^2r_4r_1, r_1(r_2r_3)^2r_1, (r_2r_3)^2, z \}, \]
since this subgroup is normal. There is an isomorphism
\begin{align*}
G_{29}/N & \to \Sigma_5, \\
r_1 N & \mapsto (1,2), \\
r_2 N & \mapsto (2,3), \\
r_3 N & \mapsto (4,5), \\
r_4 N & \mapsto (3,4).
\end{align*}
On the other hand, there is a description in \cite[Section 6]{A} of a subgroup $S$ of $G_{29}$ which is isomorphic to $\Sigma_5$ and such that the
composition $S \to GL_4(\Z^{\wedge}_5)$ is equivalent to the reduced standard representation over $\Z^{\wedge}_5$. 
The kernel of the homomorphism $S \to \Sigma_5$ is $S \cap N$, which is a $2$-group, hence elements of order five are not
in the kernel. Therefore $S \to \Sigma_5$ is an isomorphism and $G_{29}$ is a semidirect product $N \rtimes \Sigma_5$.

\begin{lemma}
\label{ReflectionSubgroupOfG29}
The subgroup $N \rtimes \Sigma_4$ is a reflection subgroup of $G_{29}$.
\end{lemma}

\begin{proof}
It is clear that the set of reflections $ X = \{ r_1,r_2,r_4,r_3r_2r_3\}$ generate the first four elements of $R$, but also
\[ z = (r_3r_2r_3r_4r_2r_1)^3, \]
hence the subgroup of $G_{29}$ generated by $X$ contains $N$ as a normal subgroup. Moreover, the image of this subgroup under
the composition $G_{29} \to G_{29}/N \to \Sigma_5$ is $\Sigma_4$, so the result follows.
\end{proof}

Since $Z(G_{29})$ is contained in $N$, the collineation group $G_{29}/Z(G_{29})$ is a semidirect product $N/Z(G_{29}) \rtimes \Sigma_5$.
The elements of $R$ have order two and their commutators belong to $Z(G_{29})$, hence $N/Z(G_{29}) \cong (\Z/2)^4$. Therefore $G_{29}/Z(G_{29})$ 
is a semidirect product $(\Z/2)^4 \rtimes \Sigma_5$ for a certain action of $\Sigma_5$ on $(\Z/2)^4$. Note that 
\[ [r_1 r_2,(r_2 r_3)^2] = \left( \begin{array}{cccc}
                              0 & -1 & 0 & 0 \\
                              -1 & 0 & 0 & 0 \\
                              0 & 0 & 0 & -1 \\
                              0 & 0 & -1 & 0 \end{array} \right) \]
does not belong to $Z(G_{29})$, hence the action of $A_5$ on $N/Z(G_{29})$ is not trivial. Therefore, the action of $\Sigma_5$ on $N/Z(G_{29})$
is faithful. By \cite[Lemma 3.2 (iii)]{Wag}, there are two conjugacy classes of faithful four-dimensional representations of $\Sigma_5$ over $\F_2$.
If we pick the basis of $(\Z/2)^4$ given by the cosets of elements of $R$, then conjugation by $r_1$ is represented by the matrix
\[ \left( \begin{array}{cccc}
          0 & 1 & 0 & 0 \\
          1 & 0 & 0 & 0 \\
          0 & 0 & 0 & 1 \\
          0 & 0 & 1 & 0 \end{array} \right), \]
which is conjugate to 
\[ \left( \begin{array}{cccc}
          1 & 0 & 0 & 0 \\
          0 & 1 & 0 & 0 \\
          1 & 0 & 1 & 0 \\
          0 & 1 & 0 & 1 \end{array} \right) \]
in $GL_4(\F_2)$, hence it is a $2$-transvection (see \cite[Section 2]{Wag} for definition of $r$-transvections). Therefore $\Sigma_5 \to GL_4(\F_2)$ is not conjugate to the standard representation
by \cite[Proof of Lemma 3.2 (iii)]{Wag}.

\begin{lemma}
\label{LevantaClasesConjugacion}
Let $q \colon G_{29}/Z(G_{29}) \to \Sigma_5$ be the quotient given by its representation as a semidirect product. For each $n=5$, $6$, there 
is a unique conjugacy class of elements of $G_{29}/Z(G_{29})$ with $q(x)$ of order $n$. 
\end{lemma}

\begin{proof}
We identify $G_{29}/Z(G_{29})$ with $(\Z/2)^4 \rtimes \Sigma_5$. Let $\sigma$ be an element of order six in $\Sigma_5$. For each $x \in (\Z/2)^4$, we have 
\[ x\sigma x^{-1} = x\sigma x^{-1} \sigma^{-1} \sigma = x \sigma x \sigma^{-1} \sigma = (x + \sigma \cdot x) \sigma = (1+\sigma)x \sigma. \]
The matrix representing $1 + \sigma$ is, up to conjugation, given by
\[ \left( \begin{array}{cccc}
          1 & 1 & 1 & 0 \\
          1 & 1 & 0 & 1 \\
          1 & 0 & 1 & 0 \\
          0 & 1 & 0 & 1 \end{array} \right), \]
which is invertible, hence $\sigma$ is conjugate to $ y \sigma$ for all $y \in (\Z/2)^4$. On the other hand, any two elements of order six in $\Sigma_5$ are conjugate,
therefore there is a unique conjugacy class of elements of the form $(y,\sigma)$ with $\sigma$ of order six in $G_{29}/Z(G_{29})$, namely, the class of $(1,2)(3,4,5)$. 
Hence, if $x \in G_{29}/Z(G_{29})$ is such that $q(x)$ has order six, then $x$ is conjugate to $(1,2)(3,4,5)$.

Similarly, if $\alpha$ is an element of order five in $\Sigma_5$, the matrix representing $1 + \alpha$ is, up to conjugation, given by
\[ \left( \begin{array}{cccc}
          1 & 1 & 1 & 0 \\
          1 & 0 & 1 & 1 \\
          0 & 1 & 1 & 1 \\
          0 & 0 & 1 & 1 \end{array} \right), \]
which is invertible, and there is a unique conjugacy class of elements of order five in $\Sigma_5$, hence the same argument applies.
\end{proof}

We now proceed to identify the torsion of $\Coker(x-1)$ for each $x \in G_{29}$.

\begin{lemma}
\label{TorsionInG29}
If $t \in S$ is an element of order five, then the torsion subgroup of $\Coker(t-1)$ is $\Z/5$. If $x \in G_{29}$ is not conjugate
to $t$, then $\Coker(x-1)$ is torsion-free.
\end{lemma}

\begin{proof}
Recall that the order of an element of $\Sigma_5$ is at most six and let $\pi \colon G_{29} \to \Sigma_5$ the quotient coming from its representation as a semidirect product. If the 
order of $\pi(x)$ is a power of $2$ or $3$, then it is conjugate in $\Sigma_5$ to an element of $\Sigma_4$. Therefore $x$ is conjugate 
to an element of $N \rtimes \Sigma_4$. By Lemma \ref{ReflectionSubgroupOfG29}, this is a reflection subgroup and its order is prime to five, hence $\Coker(x-1)$ 
is torsion-free by Remark \ref{NonModularElements}.

Assume now that $\pi(x)$ has order six and let $s$ be an element of order six in $S \leq G_{29}$. By Lemma \ref{LevantaClasesConjugacion}, we have that $x$ is conjugate to $s z^k$ for some 
$0 \leq k \leq 3$. Since $S \to GL_4(\Z_5^{\wedge})$ is equivalent to the reduced standard representation, there is a basis of $(\Z_5^{\wedge})^4$ where
\[ x(a,b,c,d) = \omega^k (b,a,-a-b-c-d,c) \]
and so the matrix of $x-1$ in this basis has the form
\[ \left( \begin{array}{cccc}
          -1 & \omega^k & 0 & 0 \\
          \omega^k & -1 & 0 & 0 \\
          -\omega^k & -\omega^k & -1-\omega^k & -\omega^k \\
          0 & 0 & \omega^k & -1 \end{array} \right), \]
whose Smith normal form is
\[ \left\{ \begin{array}{ll}
           \left( \begin{array}{cccc}
          1 & 0 & 0 & 0 \\
          0 & 1 & 0 & 0 \\
          0 & 0 & 1 & 0 \\
          0 & 0 & 0 & 0 \end{array} \right), & \text{ if $k$ is even,} \\
           & \\
          \left( \begin{array}{cccc}
          1 & 0 & 0 & 0 \\
          0 & 1 & 0 & 0 \\
          0 & 0 & 1 & 0 \\
          0 & 0 & 0 & 1 \end{array} \right), & \text{ if $k$ is odd.} \end{array} \right. \]
Therefore $\Coker(x-1)$ is torsion-free.

Let now $\pi(x)$ have order five. As in the previous case, we have that $x$ is conjugate to $t z^k$ for some $0 \leq k \leq 3$ and a fixed element $t \in S$ of 
order five. In this case, there is a basis of $(\Z_5^{\wedge})^4$ where 
\[ x(a,b,c,d) = \omega^k (-a-b-c-d,a,b,c) \]
and so the matrix of $x-1$ in this basis has the form
\[ \left( \begin{array}{cccc}
          -\omega^k-1 & -\omega^k & -\omega^k & -\omega^k \\
          \omega^k & -1 & 0 & 0 \\
          0 & \omega^k & -1 & 0 \\
          0 & 0 & \omega^k & -1 \end{array} \right), \]
whose Smith normal form is
\[ \left\{ \begin{array}{ll}
           \left( \begin{array}{cccc}
          1 & 0 & 0 & 0 \\
          0 & 1 & 0 & 0 \\
          0 & 0 & 1 & 0 \\
          0 & 0 & 0 & 1 \end{array} \right), & \text{ if $k \neq 0$,} \\
           & \\
          \left( \begin{array}{cccc}
          1 & 0 & 0 & 0 \\
          0 & 1 & 0 & 0 \\
          0 & 0 & 1 & 0 \\
          0 & 0 & 0 & 5 \end{array} \right), & \text{ if $k=0$.} \end{array} \right. \]
Hence $\Coker(x-1)$ is torsion-free if $k \neq 0$, and when $k=0$, the torsion subgroup of $\Coker(x-1)$ is $\Z/5$.
\end{proof}

This was the last piece of information needed for the main computation in this subsection.

\begin{proposition}
For each $n \geq 1$, we have
\[ |[B\Z/5^n,BX_{29}]| = \frac{1}{7680}(5^{4n}+40 \cdot 5^{3n} + 530 \cdot 5^{2n} + 2720 \cdot 5^n + 5925). \]
\end{proposition}

\begin{proof}
By Lemma \ref{TorsionInG29} and Corollary \ref{FormulaGeneral}, we have
\[ |[B\Z/5^n,BX_{29}]| = \frac{1}{7680}\Big[(3+5^n)(7+5^n)(11+5^n)(19+5^n) + 4 |G_{29}/C_{G_{29}}(t)| \Big] \]
Recall that $S \cap N = \{ 1 \}$, hence $t$ is represented in the form $(1,\alpha)$ in the semidirect product $N \rtimes \Sigma_5$
for a certain element $\alpha$ of order five. The element $(n,\sigma)$ commutes with $(1,\alpha)$ if and only if $\sigma \in C_{\Sigma_5}(\alpha) = \langle (1,2,3,4,5) \rangle$
and $\alpha \cdot n = n$. Up to conjugation, the element $\alpha$ acts on $N/Z(G_{29}) \cong (\Z/2)^4$ via the matrix
\[ \left( \begin{array}{cccc}
          0 & 1 & 1 & 0 \\
          1 & 1 & 1 & 1 \\
          0 & 1 & 0 & 1 \\
          0 & 0 & 1 & 0 \end{array} \right), \]
which has no nontrivial fixed points. Therefore $n \in Z(G_{29})$, hence $|C_{G_{29}}(t)|=20$. Thus   
\[ |[B\Z/5^n,BX_{29}]| = \frac{1}{7680}\Big[ (3+5^n)(7+5^n)(11+5^n)(19+5^n) + 4 \cdot 384 \Big]  \]
and the result follows.
\end{proof}

\begin{remark}
It can be shown that $s$ and $sz^2$ belong to reflection subgroups of $G_{29}$ of order prime to five.
\end{remark}

\subsection{The $5$-compact group $X_{31}$}

We follow the description of $G_{31}$ in Section 9 of \cite{Ben}, but using at the same time the ideas in \cite{A}. The group $G_{31}$
is generated by the generators $r_1$, $r_2$, $r_3$, $r_4$ of $G_{29}$ and the element
\[ r_5 = \left( \begin{array}{cccc}
                1 & 0 & 0 & 0 \\
                0 & 1 & 0 & 0 \\
                0 & 0 & -1 & 0 \\
                0 & 0 & 0 & 1 \end{array} \right). \]
Since $G_{31}$ contains $G_{29}$ as a reflection subgroup and we already analyzed this group in the previous subsection, 
it suffices to study the conjugacy classes of $G_{31}$ which do not intersect $G_{29}$. Let $V=\C^4$ and consider the 
composition
\[ G_{31} \to GL(V) \to PGL(V) \] 
of the representation of $G_{31}$ as a finite complex reflection group and the quotient map. The centers of $G_{29}$ and $G_{31}$ coincide, in particular, 
the center of $G_{31}$ only contains scalar matrices. Therefore we obtain an injective homomorphism
\[ G_{31}/Z(G_{31}) \to PGL(V). \]
Following \cite{A}, we compose this homomorphism with the monomorphism
\[ \Phi \colon PGL(V) \to PGL(\Lambda^2 V) \]
given by $\Phi(f)=f \wedge f$. Let $\{ e_i \mid 1 \leq i \leq 4 \}$ be the standard basis of $V$ and consider
the basis $\{ \omega_i \mid 1 \leq i \leq 6 \}$ of $\Lambda^2 V$ described in Page 31 of \cite{A}. It is straightforward
to check that the composition $G_{31} \to PGL(\Lambda^2 V)$ sends the generators $r_i$ to signed permutations of the
basis $\{ \omega_i \}$, but these signed permutations are only well defined up to multiplication by $-1$. Up to this 
multiplication, it is given by
\begin{align*}
r_1 & \mapsto (-1,1,-1,1,-1,1)(1,6)(2,3)(4,5), \\
r_2 & \mapsto (1,-1,-1,1,1,-1)(1,2)(3,5)(4,6), \\
r_3 & \mapsto (1,-1,1,1,-1,-1)(1,2)(3,6)(4,5), \\
r_4 & \mapsto (-1,1,-1,1,-1,1)(1,4)(2,3)(5,6), \\
r_5 & \mapsto (-1,1,1,-1,-1,1)(1,2)(3,4)(5,6). 
\end{align*}
Note also that these signed permutations lie in the even subgroup 
\[ H^+ = \left\{ (x_1,x_2,x_3,x_4,x_5,x_6)\sigma \in (\Z/2)^6 \rtimes \Sigma_6 \, \Bigg{|} \, \prod x_i \sgn(\sigma) = 1 \right\}, \]
where we are identifying $\Z/2 = \{1,-1\}$, so this defines a monomorphism $G_{31}/Z(G_{31}) \to H^+/Z(H^+)$. Since $G_{31}/Z(G_{31}) \cong H^+/Z(H^+)$, this is an
isomorphism. 

According to \cite{Ben}, the homomorphism $G_{29} \to \Sigma_5$ extends to a homomorphism $b \colon G_{31} \to \Sigma_6$
by sending $r_5$ to $(1,6)$, in such a way that $G_{29}$ is the inverse image of $\Sigma_5$. On the other hand, the homomorphism
$\psi \colon G_{31} \to H^+/Z(H^+)$ determines a homomorphism $q \colon G_{31} \to \Sigma_6$. These two homomorphisms are 
related through the non-inner automorphism of $\Sigma_6$ defined by
\begin{align*}
(1,2) & \mapsto (1,6)(2,3)(4,5), \\
(2,3) & \mapsto (1,2)(3,5)(4,6), \\
(4,5) & \mapsto (1,2)(3,6)(4,5), \\
(3,4) & \mapsto (1,4)(2,3)(5,6), \\
(1,6) & \mapsto (1,2)(3,4)(5,6). 
\end{align*}
As mentioned earlier, $G_{29}$ is the subgroup of $G_{31}$ of elements which map to $\Sigma_5$ under $b \colon G_{31} \to \Sigma_6$.
So the elements of $G_{31}$ which are not conjugate to elements of $G_{29}$ are those whose image under $b$ is conjugate to $(1,2,3,4,5,6)$, $(1,2,3,4)(5,6)$, $(1,2,3)(4,5,6)$
or $(1,2)(3,4)(5,6)$. By the relationship between $q$ and $b$, the elements of $G_{29}$ are those whose image under $q$ are conjugate to
$(1,3)(4,6,5)$, $(1,5)(2,6,3,4)$, $(2,5,6)$ or $(2,3)$.

In what follows, we let $\{v_i \mid 1 \leq i \leq 6\}$ be the standard basis of $(\Z/2)^6$ and say that $(x_i) \in (\Z/2)^6$ is even if $\prod x_i = 1$ and odd otherwise. Let $E$
be the subgroup of even elements in $(\Z/2)^6$.

\begin{lemma}
\label{ConjugacyCollineation31}
Let $\sigma$ be an element of $\{ (1,3)(4,6,5),(1,5)(2,6,3,4),(2,5,6),(2,3) \}$. Then $(x,\sigma)Z(H^+) \in H^+/Z(H^+)$ is conjugate to the coset of one and only one of the following
elements:
\begin{table}[h!]
\centering
\caption{}
\begin{tabular}{c | c c c}
\toprule 
$\sigma$ & & & \\
\midrule 
$(2,3)$ & $(v_1,\sigma)$ & $(v_2,\sigma)$ & $(v_1+v_2+v_4,\sigma)$ \\
\midrule 
$(2,5,6)$ & $(0,\sigma)$ & $(v_1+v_2,\sigma)$ & -  \\
\midrule
$(1,5)(2,6,3,4)$ & $(0,\sigma)$ & $(v_1+v_2,\sigma)$ & - \\
\midrule
$(1,3)(4,6,5)$ & $(v_1,\sigma)$ & $(v_2,\sigma)$ & - \\
\bottomrule
\end{tabular}
\label{table:conjugateswithsigma}
\end{table}
\end{lemma}

\begin{proof}
For simplicity, we use elements and subsets of $H^+$ to denote their corresponding images in $H^+/Z(H+)$. Note that if $(x,\sigma)$ is conjugate to $(y,\sigma)$, it must via
an element whose second coordinate lies in $C_{\Sigma_6}(\sigma)$. Given $\tau \in C_{\Sigma_6}(\sigma)$, we have
\[ (x,\tau)(r,\sigma)(\tau^{-1}(x),\tau^{-1}) = ((1+\sigma)x+\tau r,\sigma), \]
so for each $\sigma$ we have to determine the possibilities for $\tau r$ and the image of $1+\sigma$ restricted to $E$ if $\tau$ is even,
or restricted to $(\Z/2)^6-E$ if $\tau$ is odd. Since any odd element can be written in the form $v_1+y$ with $y$ even, the image of $1+\sigma$
can be obtained from the image of $1+\sigma$ restricted to $E$ and $(1+\sigma)(v_1)$. However,
\[ (1+\sigma)(v_1) = \left\{ \begin{array}{ll}
                             0, & \text{ if $\sigma = (2,3)$ or $(2,5,6)$}, \\
                             v_1+ v_5 = (1+\sigma)(v_2+v_3+v_4+v_6), & \text{ if $\sigma = (1,5)(2,6,3,4)$}, \\
                             v_1+v_3 = (1+\sigma)(v_1+v_3), & \text{ if $\sigma=(1,3)(4,6,5)$}. \end{array} \right. \]
hence the image of $1+\sigma$ equals the image of its restriction to $E$. Using that $E$ is generated by $v_1+v_j$ with $2 \leq j \leq 5$. We display
these images in Table \ref{table:imageoneplussigma}. 
\begin{table}[h]
\centering
\caption{}
\begin{tabular}{c | c }
\toprule 
$\sigma$ & $\im(1+\sigma)$ \\
\midrule 
$(2,3)$ & $\{0,v_1+v_2\}$ \\
\midrule 
$(2,5,6)$ & $\{0,v_2+v_5,v_2+v_6,v_5+v_6\}$ \\
\midrule
$(1,5)(2,6,3,4)$ & $\{ 0, v_3+v_4,v_2+v_6,v_3+v_6,v_1+v_5,v_4+v_6,v_2+v_3,v_2+v_4 \}$ \\
\midrule
$(1,3)(4,6,5)$ & $ \{  0, v_1+v_3, v_2+v_5, v_2+v_6, v_4+v_6,v_4+v_5,v_5+v_6,v_2+v_4 \}$ \\
\bottomrule
\end{tabular}
\label{table:imageoneplussigma}
\end{table}

Now for each $\sigma$, we pick $r \in (\Z/2)^6$ (odd or even, depending on the signature of $\sigma)$ and find its orbit under the action of $C_G(\sigma)$.
Then we add all elements in the image of $1+\sigma$ with elements in this orbit. This will give us the first coordinates of the conjugates of $(r,\sigma)$ 
of the form $(x,\sigma)$. We repeat with different elements of $(\Z/2)^6$ until all elements of the form $(x,\sigma)$ appear in some conjugacy class. The process is straightforward,
we summarize it Tables \ref{table:orbitcentralizer} and \ref{table:orbitsofrsigma}.
\begin{table}[h]
\centering
\caption{}
\begin{tabular}{c | c | c}
\toprule 
$\sigma$ & $r$ & $C_G(\sigma)r$ \\
\midrule 
$(2,3)$ & $v_1$ & $\{ v_1,v_4,v_5,v_6\}$  \\
 & $v_2$ & $\{ v_2,v_3\}$ \\
 & $v_1+v_2+v_4$ & $\{ v_1+v_2+v_4,v_4+v_2+v_5,v_4+v_2+v_6,v_5+v_2+v_6\}$ \\
\midrule
$(2,5,6)$ & $0$ & $\{ 0 \}$ \\
 & $v_1+v_2$ & $\{ v_i + v_j \mid i \in \{1,3,4\}, j\in \{2,5,6\} \}$ \\
\midrule
$(1,5)(2,6,3,4)$ & $0$ & $\{0\}$   \\
 & $v_1+v_2$ & $\{ v_i+v_j \mid i \in \{1,5\}, j \in \{2,3,4,6\} \}$ \\
\midrule
$(1,3)(4,6,5)$ & $v_1$ & $\{ v_1 \}$   \\
 & $v_2$ & $\{ v_2 \}$   \\
\bottomrule
\end{tabular}
\label{table:orbitcentralizer}
\end{table}
\begin{table}[h]
\centering
\caption{}
\begin{tabular}{l | l}
\toprule 
$(r,\sigma)$ & $\im(1+\sigma) + C_G(\sigma)r$ \\
\midrule 
$(v_1,(2,3))$ & \makecell{$\{ v_1,v_4,v_5,v_6,v_1+v_2+v_3,v_4+v_2+v_3,v_5+v_2+v_3,$ \\ \, $v_6+v_2+v_3\}$}  \\
\midrule
$(v_2,(2,3))$ & $\{ v_2,v_3\}$  \\
\midrule
$(v_1+v_2+v_4,(2,3))$ & \makecell{$\{ v_1+v_2+v_4,v_4+v_2+v_5,v_4+v_2+v_6,v_5+v_2+v_6,$ \\ \, $v_4+v_3+v_5,v_5+v_3+v_6\}$}  \\
\midrule
$(0,(2,5,6))$ & $\{ 0,v_2+v_5,v_2+v_6,v_5+v_6\}$  \\
\midrule
$(v_1+v_2,(2,5,6))$ & \makecell{$\{ v_1+v_2, v_3+v_2,v_4+v_2,v_1+v_5,v_3+v_5,v_4+v_5,v_1+v_6,$ \\ \, $v_3+v_6,v_4+v_6,v_1+v_3,v_1+v_4,v_3+v_4\}$}  \\
\midrule
$(0,(1,5)(2,6,3,4))$ & \makecell{$\{ 0, v_3+v_4,v_2+v_6,v_3+v_6,v_1+v_5,v_4+v_6,v_2+v_3,$ \\ \, $v_2+v_4 \}$}  \\
\midrule
$(v_1+v_2,(1,5)(2,6,3,4))$ & \makecell{$\{ v_1+v_2, v_1+v_3,v_1+v_4,v_1+v_6,v_5+v_2, v_5+v_3,v_5+v_4,$ \\ \, $v_5+v_6\}$}  \\
\midrule
$(v_1,(1,3)(4,6,5))$ & \makecell{$\{v_1, v_3, v_1+v_2+v_5, v_1+v_2+v_6, v_1+v_4+v_6,v_1+v_4+v_5,$ \\ \, $v_1+v_5+v_6,v_1+v_2+v_4\}$}  \\
\midrule
$(v_2,(1,3)(4,6,5))$ & \makecell{$\{ v_2, v_2+v_1+v_3, v_5, v_6, v_2+ v_4+v_6,v_2+ v_4+v_5,$ \\ \, $v_2+v_5+v_6,v_4 \}$}  \\
\bottomrule
\end{tabular}
\label{table:orbitsofrsigma}
\end{table}

Table \ref{table:orbitsofrsigma} is useful for future reference.
\end{proof}

Recall from the previous subsection that $G_{29}$ has a distinguished subgroup $S$ isomorphic to $\Sigma_5$ and Lemma \ref{TorsionInG29}
shows that for $x \in G_{29}$, the cokernel of $x-1$ has torsion if and only if $x$ is conjugate to an element $t$ of order five in $S$.

\begin{lemma}
\label{TorsionInG31}
If $t \in S$ is an element of order five, then the torsion subgroup of $\Coker(t-1)$ is $\Z/5$. If $x \in G_{31}$ is not conjugate
to $t$, then $\Coker(x-1)$ is torsion-free.
\end{lemma}

\begin{proof}
We need to find a lift under the quotient $\psi \colon G_{31} \to H^+/Z(H^+)$ for each conjugacy class of elements of the form $(r,\sigma)$ 
with $\sigma \in \{ (1,3)(4,6,5),(1,5)(2,6,3,4),(2,5,6),(2,3) \}$. We first find elements such that $q(x)=\sigma$ for each such $\sigma$. 
Since $q(x)=f(b(x))$ and $b$ is easier to handle, we find instead $x$ such that $b(x) = f^{-1}(\sigma)$. For simplicity, let $r_6 = r_3r_4r_2r_1r_5r_1r_2r_4r_3$, 
which satisfies $b(r_6)=(5,6)$ and 
\[\psi(r_6)=(1,1,-1,-1,-1,1)(1,5)(2,3)(4,6).\]
We summarize this step in Table \ref{table:onepreimage}. 
\begin{table}[h]
\centering
\caption{}
\begin{tabular}{c|ccc}
\toprule
$\sigma$ & $f^{-1}(\sigma)$ & $x$ & $\psi(x)$ \\
\midrule
$(2,3)$ & $(1,2)(3,4)(5,6)$ & $r_1 r_4 r_6$ & $(v_1+v_3+v_6,(2,3))$\\
\midrule
$(2,5,6)$ & $(1,2,3)(4,5,6)$ & $r_1 r_2 r_3 r_6$ & $(v_3+v_4,(2,5,6))$\\
\midrule
$(1,5)(2,6,3,4)$ & $(1,2,3,4)(5,6)$ & $r_1 r_2 r_4 r_6$ & $(v_4+v_6,(1,5)(2,6,3,4))$ \\ 
\midrule
$(1,3)(4,6,5)$ & $(1,2,3,4,5,6)$ & $r_1 r_2 r_4 r_3 r_6$ & $(v_5,(1,3)(4,6,5))$ \\
\bottomrule
\end{tabular}
\label{table:onepreimage}
\end{table}

It is convenient to find the images of the generators of $N$ under $\psi$. Let $a=r_2r_3$.
\begin{align*}
a^2 & \mapsto (1,1,-1,-1,-1,-1), \\
r_4 a^2 r_4 & \mapsto (-1,-1,1,1,-1,-1), \\
r_1 r_4 a^2 r_4 r_1 & \mapsto (-1,1,-1,-1,1,-1), \\
r_1 a^2 r_1 & \mapsto (-1,-1,1,-1,-1,1),
\end{align*}
Using these images and those of the elements in Table \ref{table:onepreimage}, it is easy to find the elements in Table \ref{table:preimagesofconjugacyclasses}.
\begin{table}[h]
\centering
\caption{}
\begin{tabular}{c|c}
\toprule
$x$ & $\psi(x)$ \\
\midrule
$r_1 r_4 r_6$ & $(v_1+v_3+v_6,(2,3))$\\
\midrule
$r_1 a^2 r_4 r_6$ & $(v_1,(2,3))$\\
\midrule
$a^2 r_1 a^2 r_4 r_6$ & $(v_2,(2,3))$ \\ 
\midrule
$r_1 r_2 r_3 r_6$ & $(v_3+v_4,(2,5,6))$ \\
\midrule
$r_4 a^2 r_4 r_1 r_2 r_3 r_6$ & $(0,(2,5,6))$ \\
\midrule
$ r_1 r_2 r_4 r_6$ & $(v_4+v_6,(1,5)(2,6,3,4))$ \\
\midrule
$r_1 r_4 a^2 r_4 r_2 r_4 r_6$ & $(v_1+v_3,(1,5)(2,6,3,4))$ \\
\midrule
$r_1 r_2 r_4 r_3 r_6$ & $(v_5,(1,3)(4,6,5))$ \\
\midrule
$a^2 r_1 r_4 a^2 r_4 r_2 r_4 r_3 r_6$ & $(v_1,(1,3)(4,6,5))$ \\
\bottomrule
\end{tabular}
\label{table:preimagesofconjugacyclasses}
\end{table}

Using Table \ref{table:orbitsofrsigma}, we see that the second column contains a representative for each 
conjugacy class of elements of the form $(r,\sigma)$ with $\sigma \in \{ (1,3)(4,6,5),(1,5)(2,6,3,4),(2,5,6),(2,3) \}$. Hence the elements
$xz^k$ with $x$ in the first column and $0 \leq k \leq 3$ form a set of representatives of all conjugacy classes of elements in $G_{31}-G_{29}$.
Now, for each $x$ in the first column, we compute the determinant of the mod $5$ reduction of $xz^k-I$ and find that it is zero only for the elements
\[ \{ r_1 r_4 r_6 z^k \mid 0 \leq k \leq 3 \} \cup \{ a^2 r_1 a^2 r_4 r_6 z^k \mid k=0,1 \} \cup \{ r_1 r_4 a^2 r_4 r_2 r_4 r_6 z^k \mid k=1,2 \}. \]
The Smith normal form of $y-I$ for each $y$ in this set shows that the cokernels are torsion-free. By Lemma \ref{TorsionInG29}, the cokernel of $x-1$
is torsion-free for any $x \in G_{29}$ which is not conjugate to $t$.
\end{proof}

We are now ready for the main computation in this subsection.

\begin{proposition}
For each $n \geq 1$, we have
\[ |[B\Z/5^n,BX_{31}]| = \frac{1}{46080}(5^{4n}+60 \cdot 5^{3n} + 1270 \cdot 5^{2n} + 11100 \cdot 5^n + 42865). \]
\end{proposition}

\begin{proof}
By Lemma \ref{TorsionInG31} and Corollary \ref{FormulaGeneral}, we have
\[ |[B\Z/5^n,BX_{31}]| = \frac{1}{46080}\Big[ (7+5^n)(11+5^n)(19+5^n)(23+5^n) + 4 |G_{31}/C_{G_{31}}(t)| \Big]. \]
It suffices to find the order of $C_{G_{31}}(t)$. The element $t$ is such that $\pi(t)$ has order five, so up to conjugation, 
it must be of the form $z^k r_1 r_2 r_4 r_3$ and  it is easy to check that $t=zr_1r_2r_4r_3$ has the desired Smith normal form. 
Since $z$ is central, we will just find the order of $C_{G_{31}}(r_1 r_2 r_4 r_3)$. Note that if $x$ commutes with $r_1 r_2 r_4 r_3$, 
then $\psi(x)$ commutes with
\[ \psi(r_1 r_2 r_4 r_3) = (v_2+v_3,(1,4,5,3,2)). \]
Let $x \in C_{G_{31}}(r_1 r_2 r_4 r_3)$ and let $\psi(x)=(n,\sigma)$. Then $\sigma$ must belong $C_{\Sigma_6}((1,4,5,3,2))$, which is
the subgroup generated by $(1,4,5,3,2)$ and there must be an equality 
\[ [(1,4,5,3,2)+1](n) = (\sigma+1)(v_2+v_3) = \left\{ \begin{array}{ll}
                                                                  0, & \text{ if } \sigma = 1, \\
																  v_1+v_3, & \text{ if } \sigma =(1,4,5,3,2), \\
                                                                  v_5+v_6, & \text{ if } \sigma =(1,5,2,4,3), \\
                                                                  v_1+v_6, & \text{ if } \sigma =(1,3,4,2,5), \\
                                                                  v_2+v_5, & \text{ if } \sigma =(1,2,3,5,4), \end{array} \right. \]
in $H^+/Z(H^+)$. Since $\sigma$ is an even permutation, we can assume that $n=0$ or $n=v_i+v_j$. A quick computations shows that 
\begin{align*}
\psi(x) \in \{ & (0,1),(v_2+v_3,(1,4,5,3,2)),(v_1+v_3,(1,5,2,4,3)),(v_4+v_3,(1,3,4,2,5)), \\
               & (v_3+v_5,(1,2,3,5,4)) \}
\end{align*} 
and therefore $|C_{G_{31}}(t)| \leq 20$. We showed in Lemma \ref{TorsionInG29} that $|C_{G_{29}}(t)|=20$, thus $|C_{G_{31}}(t)| = 20$ and therefore
\[ |[B\Z/5^n,BX_{31}]| = \frac{1}{46080}\Big[ (7+5^n)(11+5^n)(19+5^n)(23+5^n) + 4 \cdot 2304 \Big], \]
from where the result follows.
\end{proof}

\subsection{The $7$-compact group $X_{34}$}

The group $G_{34}$ has $169$ conjugacy classes, hence we use the following algorithm (see ancillary file for code) in the software GAP \cite{GAP} to achieve the
computation in this case. We phrase it for an exotic $p$-compact group $X$ corresponding to the exceptional
finite reflection group $\rho \colon W_X(T) \to \GL_l(\Z \pcom)$, since it can be used in this generality.
\begin{enumerate}
\item Determine a set $\cc(W_X(T))$ of representatives of the conjugacy classes of the mod $p$ reduction of $W_X(T)$ if $p$ is odd,
or the mod $4$ reduction if $p=2$.
\item Find the representatives whose mod $p$ reductions have nontrivial fixed points.
\item For each of the elements $x$ found in the previous step, find the kernel of $\rho'(x)-I$, where $\rho'$ is the representation of 
$W_X(T)$ as a finite complex reflection group.
\item Let $m=2$ if $p$ is odd and $m=3$ if $p=2$. For each of these elements $x$, compute $\prod p^m/j$, where $j$ runs over the elementary divisors of the mod $p^m$ 
reduction of $\rho(x)-I$ which are different from $1$.
\item Use Corollary \ref{BijectionWithLattice} and Corollary \ref{FormulaCasiGeneral} to compute $|[B\Z/p^n,BX]|$.
\end{enumerate}

Finding the conjugacy classes of the mod $p$ or mod $4$ reductions instead of $W_X(T)$ is justified by \cite[Lemma 11.3]{AGMV}. We follow steps (3) and (4)
so that we can find Smith normal forms of matrices over $\Z/p^m$ instead of $p$-adic matrices. The justification for these steps follows from Lemma \ref{AuxLemma}
and computations with GAP. Namely, the first item of Lemma \ref{AuxLemma} holds trivially for non-modular exotic $p$-compact groups and we tested in GAP that it also 
holds, with $k=1$, for any element of $G_{12}$, $G_{29}$, $G_{31}$ and $G_{34}$ at the corresponding primes where these groups
are modular. It also holds for $G_{24}$ at the prime two with $k=2$. Once this is checked, the second item of Lemma \ref{AuxLemma} is used to determine
the size of $A_w$. Recall that $A_w$ is the torsion subgroup of $\Coker(\rho(w)-I)$.

\begin{lemma}
\label{AuxLemma}
Let $w \in W_X(T)$.
\begin{enumerate}
\item If the number of elementary divisors of the mod $p^{k+1}$ reduction of $\rho(w)-I$ equals $l-\rk_{\Z \pcom}(\Ker(\rho(w)-I))$,
then the exponent $A_w$ divides $p^k$.
\item If the exponent of $A_w$ divides $p^k$, then $|A_w| = \prod p^{k+1}/j$, where $j$
runs over the elementary divisors of the mod $p^{k+1}$ reduction of $\rho(w)-I$ which are different from $1$.
\end{enumerate}
\end{lemma}

\begin{proof}
(2). Since tensor product is right exact, the cokernel of the mod $p^{k+1}$ reduction of $\rho(w)-I$ is isomorphic to 
\[ (\Z/p^{k+1})^r \oplus A_w/p^{k+1}A_w, \]
where $r$ is the $\Z \pcom$-rank of the kernel of $\rho(w)-I$, which equals the $\Z\pcom$-rank of its cokernel. By assumption, the number
of zeros in the mod $p^{k+1}$ reduction of $\rho(w)-I$ is $r$, hence the number of summands of the form $\Z/p^{k+1}$ in the
cokernel must be $r$. Since $\Z/p^n / p^{k+1}\Z/p^n$ is isomorphic to $\Z/p^{k+1}$ if $n\geq k+1$, the group $A_w$ can not
have summands $\Z/p^n$ with $n \geq k+1$ and therefore its exponent divides $p^k$. Finally, if the exponent of $A_w$ divides $p^k$, 
then 
\[ A_w/p^{k+1}A_w = A_w, \]
and the desired result follows. 
\end{proof}

Explicit generators for $G_{34}$ were deduced from \cite[Section 7]{A}, and applying the previous algorithm for this group, we obtain
\[ |[B\Z/7^k,BX_{34}]| = \frac{1}{39191040}(7^{6k}+ a_5 \cdot 7^{5k} + a_4 \cdot 7^{4k} + a_3 \cdot 7^{3k} +a_2 \cdot 7^{2k} + a_1 \cdot 7^k + a_0), \]
where $a_5=126$, $a_4=6195$, $a_3=151060$, $a_2=1904679$, $a_1=11559534$ and $a_0=31168165$.

\begin{remark}
From the computational observation that we can take $k=1$ for $X_{12}$, $X_{29}$, $X_{31}$ and $X_{34}$ at their
modular primes, and $k=2$ for $G_{24}$ at the prime two, we conclude that if $X$ is an exotic $p$-compact group corresponding 
to an exceptional finite reflection group $W_X(T) \leq GL_n(\Z \pcom)$ and $w$ belongs to a reflection subgroup 
$H$ of $W_X(T)$, then the order of the torsion subgroup of $\Coker(w-1)$ divides the order of the $p$-Sylow 
subgroup of $H$. 
\end{remark}

\bibliographystyle{amsplain}

\bibliography{mybibfile}

\providecommand{\bysame}{\leavevmode\hbox to3em{\hrulefill}\thinspace}
\providecommand{\MR}{\relax\ifhmode\unskip\space\fi MR }
\providecommand{\MRhref}[2]{%
  \href{http://www.ams.org/mathscinet-getitem?mr=#1}{#2}
}
\providecommand{\href}[2]{#2}
\begin{thebibliography}{10}

\bibitem{A}
J.~Aguad\'{e}, \emph{Constructing modular classifying spaces}, Israel J. Math.
  \textbf{66} (1989), no.~1-3, 23--40. \MR{1017153}

\bibitem{AGMV}
K.~K.~S. Andersen, J.~Grodal, J.~M. M{\o}ller, and A.~Viruel, \emph{The
  classification of {$p$}-compact groups for {$p$} odd}, Ann. of Math. (2)
  \textbf{167} (2008), no.~1, 95--210. \MR{2373153}

\bibitem{AG}
Kasper K.~S. Andersen and Jesper Grodal, \emph{The classification of 2-compact
  groups}, J. Amer. Math. Soc. \textbf{22} (2009), no.~2, 387--436.
  \MR{2476779}

\bibitem{Ben}
Mark Benard, \emph{Schur indices and splitting fields of the unitary reflection
  groups}, J. Algebra \textbf{38} (1976), no.~2, 318--342. \MR{401901}

\bibitem{BK}
A.~K. Bousfield and D.~M. Kan, \emph{Homotopy limits, completions and
  localizations}, Lecture Notes in Mathematics, Vol. 304, Springer-Verlag,
  Berlin-New York, 1972. \MR{0365573}

\bibitem{BLO2}
Carles Broto, Ran Levi, and Bob Oliver, \emph{Discrete models for the
  {$p$}-local homotopy theory of compact {L}ie groups and {$p$}-compact
  groups}, Geom. Topol. \textbf{11} (2007), 315--427. \MR{2302494}

\bibitem{DP}
F.~Destrempes and A.~Pianzola, \emph{Elements of compact connected simple {L}ie
  groups with prime power order and given field of characters}, Geom. Dedicata
  \textbf{45} (1993), no.~2, 225--235. \MR{1202101}

\bibitem{Dj1}
Dragomir~\v{Z}. Djokovi\'{c}, \emph{On conjugacy classes of elements of finite
  order in compact or complex semisimple {L}ie groups}, Proc. Amer. Math. Soc.
  \textbf{80} (1980), no.~1, 181--184. \MR{574532}

\bibitem{Dj2}
\bysame, \emph{On conjugacy classes of elements of finite order in complex
  semisimple {L}ie groups}, J. Pure Appl. Algebra \textbf{35} (1985), no.~1,
  1--13. \MR{772157}

\bibitem{DW}
W.~G. Dwyer and C.~W. Wilkerson, \emph{Homotopy fixed-point methods for {L}ie
  groups and finite loop spaces}, Ann. of Math. (2) \textbf{139} (1994), no.~2,
  395--442. \MR{1274096}

\bibitem{DW2}
\bysame, \emph{The center of a {$p$}-compact group}, The \v Cech centennial
  ({B}oston, {MA}, 1993), Contemp. Math., vol. 181, Amer. Math. Soc.,
  Providence, RI, 1995, pp.~119--157. \MR{1320990}

\bibitem{DW3}
\bysame, \emph{The fundamental group of a {$p$}-compact group}, Bull. Lond.
  Math. Soc. \textbf{41} (2009), no.~3, 385--395. \MR{2506823}

\bibitem{FS2}
Tamar Friedmann and Richard~P. Stanley, \emph{The string landscape: on formulas
  for counting vacua}, Nuclear Phys. B \textbf{869} (2013), no.~1, 74--88.
  \MR{3009223}

\bibitem{FS}
\bysame, \emph{Counting conjugacy classes of elements of finite order in {L}ie
  groups}, European J. Combin. \textbf{36} (2014), 86--96. \MR{3131877}

\bibitem{GAP}
The GAP~Group, \emph{{GAP -- Groups, Algorithms, and Programming, Version
  4.15.0}}, 2025.

\bibitem{Gon}
Alex Gonz\'{a}lez, \emph{The structure of {$p$}-local compact groups}, Ph.d.
  thesis, Universitat Aut\`onoma de Barcelona, 2010.

\bibitem{Gon2}
\bysame, \emph{The structure of {$p$}-local compact groups of rank 1}, Forum
  Math. \textbf{28} (2016), no.~2, 219--253. \MR{3466567}

\bibitem{Gro}
Jesper Grodal, \emph{The classification of {$p$}-compact groups and homotopical
  group theory}, Proceedings of the {I}nternational {C}ongress of
  {M}athematicians. {V}olume {II}, Hindustan Book Agency, New Delhi, 2010,
  pp.~973--1001. \MR{2827828}

\bibitem{Kane}
Richard Kane, \emph{Reflection groups and invariant theory}, CMS Books in
  Mathematics/Ouvrages de Math\'ematiques de la SMC, vol.~5, Springer-Verlag,
  New York, 2001. \MR{1838580}

\bibitem{L}
Jean Lannes, \emph{Sur les espaces fonctionnels dont la source est le
  classifiant d'un {$p$}-groupe ab\'{e}lien \'{e}l\'{e}mentaire}, Inst. Hautes
  \'{E}tudes Sci. Publ. Math. (1992), no.~75, 135--244, With an appendix by
  Michel Zisman. \MR{1179079}

\bibitem{N}
Morris Newman, \emph{Integral matrices}, Pure and Applied Mathematics, Vol. 45,
  Academic Press, New York-London, 1972. \MR{340283}

\bibitem{OS}
Akimou Osse and Ulrich Suter, \emph{Invariant theory and the {$K$}-theory of
  the {D}wyer-{W}ilkerson space}, Une d\'{e}gustation topologique
  [{T}opological morsels]: homotopy theory in the {S}wiss {A}lps ({A}rolla,
  1999), Contemp. Math., vol. 265, Amer. Math. Soc., Providence, RI, 2000,
  pp.~175--185. \MR{1803957}

\bibitem{PW}
A.~Pianzola and A.~Weiss, \emph{The rationality of elements of prime order in
  compact connected simple {L}ie groups}, J. Algebra \textbf{144} (1991),
  no.~2, 510--521. \MR{1140619}

\bibitem{San}
Marta Santos, \emph{Brown-{P}eterson cohomology and {M}orava {$K$}-theory of
  {${\rm DI}(4)$} and its classifying space}, Fund. Math. \textbf{162} (1999),
  no.~3, 209--232. \MR{1736361}

\bibitem{Serre}
Jean-Pierre Serre, \emph{Groupes finis d'automorphismes d'anneaux locaux
  r\'eguliers}, Colloque d'{A}lg\`ebre ({P}aris, 1967), \'Ecole normale
  sup\'erieure de jeunes filles, Secr\'etariat math\'ematique, Paris, 1968,
  pp.~Exp. 8, 11. \MR{234953}

\bibitem{Sm}
Larry Smith, \emph{Polynomial invariants of finite groups}, Research Notes in
  Mathematics, vol.~6, A K Peters, Ltd., Wellesley, MA, 1995. \MR{1328644}

\bibitem{Smith}
\bysame, \emph{Lannes {$T$}-functor and invariants of pointwise stabilizers},
  Forum Math. \textbf{12} (2000), no.~4, 461--476. \MR{1763901}

\bibitem{SS}
Larry Smith and R.~M. Switzer, \emph{Realizability and nonrealizability of
  {D}ickson algebras as cohomology rings}, Proc. Amer. Math. Soc. \textbf{89}
  (1983), no.~2, 303--313. \MR{712642}

\bibitem{S}
Ronald Solomon, \emph{Finite groups with {S}ylow {$2$}-subgroups of type
  {$3$}}, J. Algebra \textbf{28} (1974), 182--198. \MR{0344338}

\bibitem{Springer}
T.~A. Springer, \emph{Regular elements of finite reflection groups}, Invent.
  Math. \textbf{25} (1974), 159--198. \MR{354894}

\bibitem{Wag}
Ascher Wagner, \emph{The faithful linear representation of least degree of
  {$S\sb{n}$} and {$A\sb{n}$} over a field of characteristic {$2.$}}, Math. Z.
  \textbf{151} (1976), no.~2, 127--137. \MR{419581}

\end{thebibliography}

\end{document}